\documentclass[a4paper, 11p]{amsart}
\usepackage[a4paper, total={6in, 8in}]{geometry}

\usepackage{tabu}

\usepackage{amsmath}
\usepackage{amssymb}
\usepackage{amsthm}
\usepackage{bbm}
\usepackage{makecell}

\author{E. Assing}
\title{A note on Sarnak's density hypothesis for $\textrm{Sp}_4$}

\theoremstyle{plain}
\newtheorem{lemma}{Lemma}[section]
\newtheorem{theorem}[lemma]{Theorem}
\newtheorem{conjecture}[lemma]{Conjecture}
\newtheorem{remark}[lemma]{Remark}

\DeclareMathOperator{\R}{\mathbb{R}}

\DeclareMathOperator{\A}{\mathbb{A}}
\DeclareMathOperator{\Q}{\mathbb{Q}}
\DeclareMathOperator{\Z}{\mathbb{Z}}

\DeclareMathOperator{\PGSp}{\text{PGSp}_4}
\DeclareMathOperator{\GSp}{\text{GSp}_4}
\DeclareMathOperator{\Sp}{\text{Sp}_4}

\begin{document}

\thanks{The author is supported by the Germany Excellence Strategy grant EXC-2047/1-390685813 and also partially by the DFG – Project-ID 491392403 – TRR 358}

\setcounter{tocdepth}{2}  \maketitle 

\begin{abstract}
In this note we prove Sarnak's (spherical) density hypothesis for the full discrete spectrum of the quotients $\Gamma_{\textrm{pa}}(q)\backslash \textrm{Sp}_4(\mathbb{R})$, where $\Gamma_{\textrm{pa}}(q)$ are paramodular groups with square-free level $q$. To derive this estimate we upgrade a density estimate established via the Kuznetsov formula, which only accounts for the generic part, using Arthur's parametrization of the discrete spectrum.
\end{abstract}

\section{Introduction}

The generalized Ramanujan conjecture is a central and very difficult problem in the theory of automorphic forms. For ${\rm GL}_n$ it predicts that all cuspidal representations are tempered. It has been known for a long time that this does not directly generalize to other groups, see \cite{Sa4}. The known counter examples to the naive Ramanujan conjecture are all cuspidal associated to parabolics (in short CAP). The correct generalized Ramanujan conjecture now states that all cuspidal non CAP forms are tempered. But even for ${\rm GL}_2$, where no CAP forms appear, the Ramanujan conjecture is still out of reach in general. Therefore, in practice, one often tries to replace the full conjecture by suitable approximations. One such approximation is Sarnak's density hypothesis, which predicts in a precise quantitative way that exceptions to the generalized Ramanujan conjecture are rare. This density hypothesis has the advantage that it can be formulated for the full discrete spectrum and it is expected to hold uniformly in great generality. In particular, it includes both CAP forms and residual forms. As we will see below this feature is of particular importance in the context of this note.

\subsection{The main result}

Before continuing our discussion let us formulate the spherical density hypothesis for ${\rm Sp}_4$. Let $M>2$ be a (large) parameter and $\Gamma\subseteq {\rm Sp}_4(\mathbb{Q})$ be a congruence lattice. To each eigenform of all invariant differential operators in $L^2_{\rm disc}(\Gamma\backslash \mathbb{H}_2)$ we associate a spectral parameter $\mu_{\varpi}=(\mu_{\varpi}(1),\mu_{\varpi}(2))\in \mathbb{C}^2$. Let $\mathcal{F}_{\Gamma}(M)$ denote a maximal orthogonal family of eigenforms $\varpi$ with $\Vert \mu_{\varpi} \Vert\leq M$. The restriction of the spectral parameter to a ball of radius $M$ is important to make the set $\mathcal{F}_{\Gamma}(M)$ finite. Note that the constant function $\mathbf{1}$ has spectral parameter
\begin{equation}
	\mu_{\mathbf{1}}=(\frac{3}{2},\frac{1}{2}) \nonumber
\end{equation}
and is always in $\mathcal{F}_{\Gamma}(M)$. We have normalized the spectral parameters so that an eigenform $\varpi$ is tempered if and only if $\mu_{\varpi}\in (i\mathbb{R})^2$. Thus we introduce
\begin{equation}
	\sigma_{\varpi} = \max_{j=1,2}(\vert\Re(\mu_{\varpi}(j))\vert) \nonumber
\end{equation}
to measure how far an eigenfunction is from being tempered. Possible exceptions of badness $\sigma$ are counted by
\begin{equation}
	N_{\Gamma}(\sigma;M) = \sharp\{\varpi \in \mathcal{F}_{\Gamma}(M)\colon \sigma_{\varpi}\geq\sigma \}.\nonumber
\end{equation}
The two extreme cases are
\begin{equation}
	N_{\Gamma}(0;M) = \sharp \mathcal{F}_{\Gamma}(M) \text{ and }N_{\Gamma}(\sigma_{\mathbf{1}};M) = 1.\nonumber
\end{equation}
Interpolating these two bounds leads to the density hypothesis
\begin{equation}
	N(\sigma;M) \ll \sharp\mathcal{F}_{\Gamma}(M)^{1-\frac{\sigma}{\sigma_{\mathbf{1}}}}.\nonumber
\end{equation}
We will be mainly interested in varying the lattice $\Gamma$ through some suitable level family and will therefore ignore the dependency on $M$ confining ourselves to the weaker hypothesis
\begin{equation}
	N_{\Gamma}(\sigma;M) \ll_{M,\epsilon} {\rm Vol}(\Gamma\backslash\mathbb{H}_2)^{1-\frac{\sigma}{\sigma_{\mathbbm{1}}}+\epsilon}. \label{density_hypothesis}
\end{equation}
Our main result is Theorem~\ref{th:detail_main} below. It establishes the density hypothesis and more for the paramodular group $\Gamma_{\textrm{pa}}(q)$ as defined in \eqref{para_def} below. We can summarize the most important part of the statement as follows:

\begin{theorem}\label{main_theorem}
For $q$ square-free, we have
\begin{equation}
	N_{\Gamma_{\textrm{pa}}(q)}(\sigma;M) \ll_{M,\epsilon} {\rm Vol}(\Gamma_{\textrm{pa}}(q)\backslash\mathbb{H}_2)^{1-\frac{\sigma}{\sigma_{\mathbbm{1}}}(1+\frac{1}{2})+\epsilon}+1. \label{main_th_eq}
\end{equation}
\end{theorem}

\subsection{The methods}

The proof of Theorem~\ref{main_theorem} has two main steps. First, a density result for generic forms is proved. The main global tool here is a Kuznetsov type formula for $\textrm{Sp}_4$, which has the restriction to generic forms built in by default. A similar argument has appeared in \cite{Man} and we follow it closely.\footnote{Unfortunately the argument in \cite{Man} seems to have a gap. We explain this in more detail in Remark~\ref{mistake} below.}  Second, we have to account for the non-generic forms manually. This is done using Arthur's classification of the discrete spectrum predicted in \cite{Ar1}. We will see that non-generic non-CAP cusp forms (i.e. those cusp forms that are expected to be tempered) are covered by the generic density result. The remaining task is to handle the contribution of CAP forms. These are treated by carefully studying the associate Arthur packets, which were computed explicitly in \cite{Sc2}. 

It should be mentioned that our approach to use Arthur's classification to control the non-generic spectrum is not completely new. Indeed, similar ideas were used in \cite{MS} and later also in \cite{EGG}.

Note that one might have hoped to transfer the density estimates directly from ${\rm GL}_4$ to ${\rm Sp}_4$. However, since the image of the transfer from ${\rm Sp}_4$ is self-dual and the Plancherel-Density drops when restricted to self-dual forms, this turns out to be difficult. The latter phenomenon was worked out in \cite{Ka}.

\subsection{Further discussion}

The density hypothesis in rank one has been a hot topic in the 90's and is relatively well understood. We refer to \cite{Hum, Hux, Iw, Sa1} and the references within for a more in depth discussion of the matter. In higher rank recently breakthroughs have been obtained for ${\rm GL}_n$, where \eqref{density_hypothesis} has been established for the standard Hecke congruence subgroup, the principal congruence subgroup and a Borel-Type congruence subgroup, see \cite{Bl, AB, A}. Note that even though we have formulated the density hypothesis only for ${\rm Sp}_4$, it should be clear how to adapt this formulation to other groups. For groups of higher rank different from $\textrm{GL}_n$ strong density results are still sparse. Following partial progress in the case of ${\rm Sp}_4$ made in \cite{Man} our result is, to the best of our knowledge, the first spherical density theorem covering the full discrete spectrum beyond $\textrm{GL}_n$. We hope that the strategy of combining a generic density hypothesis, obtained using a version of the Kuznetsov formula, with the endoscopic classification will generalize to a wider class of groups.

We conclude the introduction by briefly pointing towards some interesting points:

\begin{itemize}
	\item We have made no attempt of explicating how the implicit constant in \eqref{main_th_eq} depends on $M$. A close analysis of the argument along the lines of \cite{AB} will show that this dependency is polynomial, but obtaining a reasonable exponent would require more work. 
	
	\item The assumption that $q$ is square-free enters only in the first step of our proof. Indeed, when setting up the Kuznetsov formula we make use of \cite{CI} to handle the spectral side. This is where square-freeness is required. The general case would follow from a conjecture of Lapid-Mao \cite{LM} (see also Conjecture~\ref{conj} below) together with the corresponding local estimates \eqref{con_2}. Note that the conjecture of Lapid and Mao is known for tempered representations due to recent work by M. Furusawa and K. Morimoto, see \cite[Theorem~6.3]{FM}. Unfortunately one would need the conjecture precisely for (the conjecturally empty set of) generic non-tempered forms.
	
	\item When comparing the estimate given in Theorem~\ref{main_theorem} with Sarnak's density hypothesis as stated in \eqref{density_hypothesis} we observe that there is an improvement of $\frac{1}{2}$ in the exponent. This breach of the density barrier has its roots in the estimate for the generic contribution, which is obtained using the Kuznetsov formula. The relevant part of the argument is carried out in Section~\ref{sec:gen} below. One can not expect to obtain further improvements in the density theorem as stated in \eqref{main_th_eq}, because the estimate is sharp at $\sigma=\frac{1}{2}$, where CAP-forms of Saito-Kurokawa type contribute. However, one can expect further improvements in the exponent on the generic part of the spectrum. This requires non-trivial estimates for Kloosterman sums and we will explore this feature for prime $q$. See Theorem~\ref{th:gen_ests} and Theorem~\ref{th:detail_main} below for details.
	
	\item Other approximations to the generalized Ramanujan conjecture are absolute upper bounds on $\sigma_{\varpi}$ for non-CAP forms $\varpi$. In the case of ${\rm Sp}_4$ these can be obtained by exploiting the transfer to ${\rm GL}_4$ established in \cite{Ar1, Ar2} (or \cite{AS} in the generic case) and applying bounds towards the Ramanujan conjecture on $\textrm{GL}_4$. The latter are well recorded in \cite{BB} for example. The CAP forms of $\Sp$ all satisfy $\sigma_{\varpi} = \frac{1}{2}$, so that an in principle non negligible part of the spectrum is heavily non-tempered. However, these forms are still covered by Theorem~\ref{main_theorem}. This is not only convenient for possible applications, but also shows how robust Sarnak's density hypothesis is.
	
	\item There are different measures for the non-temperedness of $\varpi$, which lead to alternative formulations of the density hypothesis. In general not all of these are equivalent. For applications where the pre-trace formula is used directly our formulation works quite well. Examples that fall in this category are lattice point counting problems and optimal lifting. See \cite{JK} or \cite[Section~1.3 and Section~8]{AB} for related discussion in the context of $\textrm{GL}_n$. Another popular formulation is in terms of $p(\varpi)$, which is the infimum over real numbers $p\geq 2$ such that the matrix coefficient of $\varpi$ is in $L^p(\textrm{Sp}_4(\R))$. It is this invariant, which was used by Sarnak and Xue to formulate their multiplicity conjecture, see \cite{SX}. This aspect is considered in \cite{EGG} for spaces closely related to ours. One can deduce the value of $p(\varpi)$ from the information contained in $\Re(\mu_{\varpi})$, but in general not from $\sigma_{\varpi}$ alone. See \cite[Lemma~3.2]{GGN} for example.

	\item Our arguments apply in principle to other lattice families, but some difficulties arise in general. Indeed, the paramodular group is convenient for several reasons. First, we have the the results from \cite{CI} available. This makes the generic estimated unconditional for square-free levels. Second, as shown in \cite{Sc1, Sc2, RS} for example, the level structure given by the paramodular group has many nice properties. In particular, it harmonizes with Arthur's classification. That said, we will take the liberty to comment on features that arise for other lattice families once in a while.
\end{itemize}

\section{Notation and Preliminaries}

As the title suggests we are interested in establishing (a version of) Sarnak's density hypothesis for $\Sp$. As the mostly classical discussion in the introduction reveals we are particularly interested in the spectrum of (congruence) quotients of the form $\Gamma\backslash \mathbb{H}_2$ for $\Gamma\subseteq \Sp(\Q)$. However, in order to establish the desired result it will be useful to work adelically. It turns out that in the latter setting it is often more convenient to work with $\GSp$ or $\PGSp$ instead of $\Sp$. Therefore it will be important to introduce all these groups as well as some related objects.

Conventions differ when dealing with symplectic groups. We will closely follow the set up in \cite{Man}. 

\subsection{Matrices and subgroups}

We set 
\begin{equation}
	J= \left(\begin{matrix} 0&0&1&0 \\0&0&0&1\\ -1&0&0&0\\ 0&-1&0&0\end{matrix}\right).\nonumber
\end{equation}
The general symplectic group over a ring $R$ (with unit) is given by
\begin{equation}
	G(R) = {\rm GSp}_4(R) = \{g\in{\rm Mat}_4(R)\colon g^{\top}Jg=\mu \cdot J \text{ for }\mu=\mu(g)\in R^{\times} \}.\nonumber
\end{equation}
The map $g\mapsto \mu(g)$ is called the multiplier homomorphism. The special symplectic group is denoted by $G_0(R) = {\rm Sp}_4(R)$ and consists of those $g\in G(R)$ with $\mu(g)=1$. Furthermore we set $G(\mathbb{R})^+ = \{g\in G(\mathbb{R})\colon \mu(g)>0 \}.$

Our choice of maximal torus is
\begin{equation}
	T(R) = \{ {\rm diag}(x,y,\mu\cdot x^{-1}, \mu\cdot y^{-1})\colon x,y,\mu\in R^{\times}  \} \nonumber
\end{equation}
and we put $T_0(R) =T(R) \cap {\rm Sp}_4(R)$. For later reference we write $t(x,y) = {\rm diag}(x,y,x^{-1},y^{-1})$. A useful variant is given by 
\begin{equation}
	c=(c_1,c_2) \mapsto c^{\ast} = t\left(\frac{1}{c_1},\frac{c_1}{c_2}\right). \nonumber
\end{equation}

Over the real numbers the maximal compact subgroup of is given by $K_{\infty}={\rm O}_4(\mathbb{R})\cap G(\mathbb{R})$. The classical Siegel upper half space (of degree $2$) can be identified with
\begin{equation}
	\mathbb{H}_2 = G_0(\mathbb{R})/(K_{\infty}\cap G_0(\mathbb{R})). \nonumber
\end{equation} 
We will further need the following special coordinates in the real case. We first define the embedding $\iota\colon \mathbb{R}_+^2\to T_0(\mathbb{R})$ by 
\begin{equation}
	\iota(y) = t(y_1y_2^{\frac{1}{2}},y_2^{\frac{1}{2}}). \nonumber
\end{equation}
The inverse of $\iota$ will be denoted by ${\rm y}\colon T_0(\mathbb{R}_+) \to \mathbb{R}_+^2$. We extend ${\rm y}$ to $G_0$ via the Iwasawa decomposition $G_0(\mathbb{R}) = U(\mathbb{R})T_0(\mathbb{R}_+)(K_{\infty}\cap G_0)$. 

Next we define some unipotent matrices:
\begin{equation}
	n(x) = \left(\begin{matrix} 1&x&0&0\\ 0&1&0&0\\0&0&1&0\\0&0&-x&1\end{matrix}\right) \text{ and } s\left(\left(\begin{matrix} a& b \\ b & c \end{matrix}
	\right)\right) = \left(\begin{matrix}1&0&a&b \\0&1&b&c\\0&0&1&0\\0&0&0&1\end{matrix}\right).\nonumber
\end{equation}
The Borel subgroup is given by
\begin{equation}
	B=TU=UT \text{ with }U(R) = \left\{ n(x)\cdot s(M)\colon x\in R \text{ and }M\in {\rm Sym}_{2\times 2}(R) \right\}.\nonumber
\end{equation}
This is the standard minimal parabolic subgroup. 

Let $\Gamma_0={\rm Sp}_4(\mathbb{Z})$ and define the paramodular group of level $q$ by
\begin{equation}
	\Gamma_{\rm pa}(q) = \left[\begin{matrix}\mathbb{Z} & \mathbb{Z} & q^{-1}\mathbb{Z} & \mathbb{Z} \\ q\mathbb{Z}&\mathbb{Z}&\mathbb{Z}& \mathbb{Z} \\ q\mathbb{Z}&q\mathbb{Z}&\mathbb{Z}&q\mathbb{Z}\\q\mathbb{Z}&\mathbb{Z}&\mathbb{Z}&\mathbb{Z}\end{matrix} \right]\cap {\rm Sp}_4(\mathbb{Q}).\label{para_def}
\end{equation}
Note that this is not a subgroup of $\Gamma_0$, but $U(\mathbb{Z})\subseteq \Gamma\cap U(\mathbb{Q})$. The latter is not an equality and the difference is measured by the index
\begin{equation}
	\mathcal{N}(\Gamma_{\rm pa}(q)) = [\Gamma_{\rm pa}(q)\cap U(\mathbb{Q})\colon U(\mathbb{Z})]=q.\label{eq:N_for_para}
\end{equation}
While our main results concern only the paramodular group $\Gamma_{\rm pa}(q)$, we will often discuss general aspects of the argument for arbitrary lattices $\Gamma\subseteq \Sp(\Q)$. This allows us to comment on certain interesting features that arise for the principal congruence subgroup or other parahoric lattices. We write $X_{\Gamma}$ for the classical quotient
\begin{equation}
	X_{\Gamma} = \Gamma\backslash \mathbb{H}_2.\nonumber
\end{equation}

When working adelically we will encounter a certain open compact subgroup as counterparts to the lattice $\Gamma$. Recall that strong approximation for $G$ states that
\begin{equation}
	G(\mathbb{A}) = G(\mathbb{Q})\cdot (G(\mathbb{R})^+\times K) \nonumber
\end{equation}
where $K\subseteq G(\mathbb{A}_{\rm fin})$ is an open compact subgroup such that
\begin{equation}
	\mu\colon K\to \widehat{\mathbb{Z}}^{\times} \nonumber
\end{equation}
is surjective. We thus define the groups $K_{\Gamma}$ by requiring
\begin{equation}
	\Gamma = G(\mathbb{Q})\cap G(\mathbb{R})^+\times K_{\Gamma}. \nonumber
\end{equation}
Then $K_{\Gamma} = \prod_p K_{\Gamma,p}$ and assume that $K_{\Gamma,p}\subseteq G(\mathbb{Q}_p)$ is an open compact subgroup. For the paramodular group it is easy to write down these groups explicitly:
\begin{equation}
	K_{\Gamma_{\rm pa}(q),p} = \left[\begin{matrix}
	\mathbb{Z}_p & \mathbb{Z}_p & p^{-v_p(q)}\mathbb{Z}_p & \mathbb{Z}_p \\ p^{v_p(q)}\mathbb{Z}_p&\mathbb{Z}_p&\mathbb{Z}_p& \mathbb{Z}_p \\ p^{v_p(q)}\mathbb{Z}_p&p^{v_p(q)}\mathbb{Z}_p&\mathbb{Z}_p&p^{v_p(q)}\mathbb{Z}_p\\ p^{v_p(q)}\mathbb{Z}_p&\mathbb{Z}_p&\mathbb{Z}_p&\mathbb{Z}_p	\end{matrix}\right]\cap G(\mathbb{Q}_p).\nonumber
\end{equation}
In general one sees that there is a square-free integer $D_{\Gamma}$ so that $K_{\Gamma,p} = G(\Z_p)$ for all $p\mid D_{\Gamma}$. For lack of a better name we refer to $D_{\Gamma}$ as the discriminant of $\Gamma$. Note that if $\Gamma = \Gamma_{\rm pa}(q)$, then $D_{\Gamma} = \textrm{rad}(q)$. Of course one has $D _{\Gamma_0} = 1$.

Note that, if not explicitly stated otherwise, then rational matrices (or scalars) will always be embedded diagonally in the corresponding adelic spaces. On the other hand we will use the notation $\iota_{\infty}\colon G(\mathbb{R})\to G(\mathbb{A})$ for the embedding $g\mapsto (g,1,\ldots)$. This map will play a distinguished role in the adelisation process of automorphic forms and therefore deserves a special name.

The two simple roots of $G_0$ are given by
\begin{equation}
	\alpha(t(x,y))=xy^{-1} \text{ and } \beta(t(x,y))=y^2.\nonumber
\end{equation}
The positive roots are then $\Sigma_+ = \{\alpha,\beta,\alpha+\beta,2\alpha+\beta\}$.

We denote the Weyl group by $W$. It is given by
\begin{equation}
	W=\{ 1,s_1,s_2,s_1s_2, s_2s_1, s_1s_2s_1, s_2s_1s_2, s_1s_2s_1s_2 \},\nonumber
\end{equation}
for 
\begin{equation}
	s_1=\left(\begin{matrix} 0& 1 & 0 & 0\\ -1 & 0 & 0 & 0\\ 0&0&0&1\\0&0&-1&0\end{matrix}\right) \text{ and } s_2= \left(\begin{matrix} 1 & 0&0&0\\ 0&0&0&1\\0&0&1&0\\0&-1&0&0\end{matrix}\right).\nonumber
\end{equation}
Note that $s_1$ (resp. $s_2$) is the simple root reflection corresponding to $\alpha$ (resp. $\beta$). We set $U_w=w^{-1}U^{\top}w\cap U$ and $\overline{U}_w = w^{-1}Uw\cap U$. Further we will write ${}^wy = {\rm y}[w\iota(y)^{-1}w^{-1}]$ for $y\in \mathbb{R}_2^+$.

\subsection{Measures and characters}

We first equip $\mathbb{R}$ with the usual Lebesgue measure. The fields $\mathbb{Q}_2,\mathbb{Q}_3,\ldots$ are equipped with the additive Haar measure normalized by ${\rm Vol}(\mathbb{Z}_p)=1$. These give rise to the product measure on $\mathbb{A}$, which by strong approximation satisfies ${\rm Vol}(\mathbb{Q}\backslash \mathbb{A}) = 1$. If $R$ is one of these rings, we put the Tamagawa measure on $U(R)$. It is normalised by ${\rm Vol}(U(\mathbb{Q})\backslash U(\mathbb{A}))=1$. Furthermore, the measures on $U$ naturally factor through $U=U_w\cdot \overline{U}_w$.

For $\alpha\in \mathbb{C}^2$ and $y\in \mathbb{R}_+^2$ we employ the usual notation $y^{\alpha}=y_1^{\alpha_1}\cdot y_2^{\alpha_2}$. Important will be the exponent $\eta=(2,\frac{3}{2})$, which is used to define the measure
\begin{equation}
	d^{\ast}y = y^{-2\eta}\frac{dy_1}{y_1}\cdot \frac{dy_2}{y_2} \nonumber
\end{equation}
on $\mathbb{R}_+^2$. We use $\iota$ to push this measure to $T_0(\mathbb{R}_+)$. Finally we equip $K_{\infty}$ with the Haar measure such that ${\rm Vol}(K_{\infty}\cap G_0(\mathbb{R}))=1$. The Iwasawa decomposition allows us to define a Haar measure on $G_0(\mathbb{R})$ by
\begin{equation}
	\int_{G_0(\mathbb{R})}f(g)dg = \int_{K_{\infty}\cap G_0(\mathbb{R})}\int_{T_0(\mathbb{R}_+)}\int_{U(\mathbb{R})} f(u\iota(y)k)dud^{\ast}ydk.\nonumber
\end{equation}
This descends to the usual measure $dud^{\ast}y$ on $\mathbb{H}_2$. To make the comparison of different measures easier we recall that
\begin{equation}
	{\rm Vol}(G_0(\mathbb{Z})\backslash G_0(\mathbb{R}),dg) = \frac{\zeta(2)\zeta(4)}{2\pi^3}.\nonumber
\end{equation}
See for example \cite[Proposition~A.3]{KL}. We can use this measure to define a measure on ${\rm PGSp}_4(\mathbb{R})$ and to normalize $G(\R)$ in the obvious way.

At the finite places we normalize the Haar measure on $Z(\mathbb{Q}_p)\backslash G(\mathbb{Q}_p)$ and $G(\mathbb{Q}_p)$ so that 
\begin{equation}
	{\rm Vol}((Z(\mathbb{Q}_p)\cap G(\mathbb{Z}_p))\backslash G(\mathbb{Z}_p),d\overline{g}) = 1 = {\rm Vol}(G(\mathbb{Z}_p),dg).\nonumber
\end{equation}
(Note that these normalizations are compatible if we equip $\mathbb{Q}_p^{\times}$ with the Haar measure satisfying ${\rm Vol}(\mathbb{Z}_p^{\times},d^{\times}x)=1$.) By \cite[Lemma~3.3.3]{RS} we have
\begin{equation}
	{\rm Vol}(K_{\Gamma_{\rm pa}(q),p},dg) = q^{-2}(1+q^{-2})^{-1}.\nonumber
\end{equation}

Globally we equip $G(\mathbb{A})$ with the product measure $dg=dg_{\infty}\cdot\prod_pdg_p$. Using strong approximation it is easy to see that with our normalizations we have
\begin{equation}
	{\rm Vol}(Z(\mathbb{A})G(\mathbb{Q})\backslash G(\mathbb{A}),dg) = \frac{\zeta(2)\zeta(4)}{2\pi^3}.\nonumber
\end{equation}
For comparison let us note that the Tamagawa measure on $G(\mathbb{A})$, which we will denote by $d^{\rm ta}g$, is normalized such that 
\begin{equation}
	{\rm Vol}(Z(\mathbb{A})G(\mathbb{Q})\backslash G(\mathbb{A}),d^{\rm ta}g) = 2.\nonumber
\end{equation}
Thus, we must have $d^{\rm ta}g = \frac{4\pi^3}{\zeta(2)\zeta(4)}dg$. 

Given a lattice $\Gamma$ we write
\begin{equation}
	\mathcal{V}(\Gamma) = {\rm Vol}(X_{\Gamma}).\nonumber
\end{equation}
for the co-volume. It turns out that $\mathcal{V}(\Gamma_{\rm pa}(q))\asymp q^2$.

We now fix the standard additive, $\mathbb{Q}$-invariant character $\psi\colon \mathbb{Q}\backslash \mathbb{A}\to S^1$. We have a factorization $\psi=\psi_{\infty}\cdot \prod_p \psi_p$ with $\psi_{\infty}(x) =e(x)$. Also note that all local character $\psi_p$ are trivial on $\mathbb{Z}_p$ but non-trivial on $p^{-1}\mathbb{Z}_p$ (i.e. they are unramified). This gives rise to a character $$\boldsymbol{\psi}\colon U(\mathbb{Q})\backslash U(\mathbb{A}) \to S^1$$ given by
\begin{equation}
	\boldsymbol{\psi}^{(X)}\left(n(x)s\left(\left(\begin{matrix} \ast & \ast \\ \ast & c\end{matrix}\right)\right)\right) = \psi(X_1x+X_2c),\nonumber
\end{equation}
for $X=(X_1,X_2)\in (\mathbb{A}^{\times})^2$. Of course we have the factorization $\boldsymbol{\psi}^{(X)} = \boldsymbol{\psi}_{\infty}^{(X_{\infty})}\cdot \prod_p \boldsymbol{\psi}_p^{(X_p)},$ where the local characters are defined in the obvious way. If $X$ is trivial, we drop it from the notation. The archimedean character $\boldsymbol{\psi}_{\infty}$ on $U(\mathbb{R})$ will play a distinguished role below in our discussion of Kloosterman sums.

\section{Kloosterman sums}\label{sec_KS}

We will now introduce Kloosterman sums following \cite{Man0, Man}, where the relevant theory is worked out explicitly for ${\rm Sp}_4$. A more general discussion can be found in \cite{Da, Dr2}. Throughout this discussion it will be important that the lattice $\Gamma$ satisfies $U(\mathbb{Z})\subseteq \Gamma\cap U(\mathbb{Q})$. Recall that this is true for $\Gamma_{\rm pa}(q)$.

\subsection{Definitions and basic properties}

We start by defining the global Kloosterman set
\begin{equation}
	\mathfrak{X}_{\Gamma}(c^{\ast}w) = U(\mathbb{Z})\backslash [U(\mathbb{Q})wc^{\ast}U_w(\mathbb{Q})\cap \Gamma]/U_w(\mathbb{Z}).\nonumber
\end{equation}

For $M,N\in \mathbb{N}^2$ we call a tuple $(w,c)\in W\times \mathbb{N}^2$ admissible (or more precisely $(M,N)$-admissible) if
\begin{equation}
	\boldsymbol{\psi}_{\infty}^{(M)}(wc^{\ast}x(c^{\ast})^{-1}w^{-1})=\boldsymbol{\psi}_{\infty}^{(N)}(x) \text{ for all }x\in \overline{U}_w.\nonumber
\end{equation}
This is precisely the condition \cite[(4.1)]{Man}, which is termed \emph{well-definedness condition} there. The constraints imposed on $c=(c_1,c_2)$ by the admissibility of $(w,c)$ can be computed by hand. Details can be found in \cite[(4.4)]{Man} or \cite[Section~5.1]{Man0} (and also in slightly different form in \cite[Section~5.4]{Co}). We summarize the constraints on $N,M$ and $c$ arising from admissibility in Table~\ref{tab:my_label}.
\begin{table}
	\centering
	\begin{tabular}{ | c | c | c | c | c | }
		\hline
		$w$ & 1 & $s_1$ & $s_2$ & $s_1s_2$   \\
		\hline 
		$c=(c_1,c_2)$ & $M=N$ &  $N_2=M_2=0$ & $N_1=M_1=0$ & $N_1=M_2=0$ \\
		\hline 
		\hline
		$w$ & $s_2s_1$  & $s_1s_2s_1$ & $s_2s_1s_2$ & $s_1s_2s_1s_2$ \\
		\hline
		$c=(c_1,c_2)$ & $N_2=M_1=0$   & $N_2=M_2\cdot \frac{c_1^2}{c_2^2}$ & $N_1=M_1\cdot \frac{c_2}{c_1^2}$ & - \\\hline 
	\end{tabular}
	\caption{Admissibility constraints.}
	\label{tab:my_label}
\end{table}

Further, we call a Weyl element $w$ relevant, if the condition on $c$ can be satisfied without $N_1,M_1,N_2$ or $M_2$ being zero. Thus the relevant Weyl elements are $1, s_1s_2s_1, s_2s_1s_2$ and $s_1s_2s_1s_2$.

For admissible $(w,c)$ we define the Kloosterman sum by
\begin{equation}
	{\rm KL}_{\Gamma,w}(c;M,N) = \sum_{xwc^{\ast}x'\in \mathfrak{X}_{\Gamma}(c^{\ast}w)} \boldsymbol{\psi}_{\infty}^{(M)}(x)\boldsymbol{\psi}_{\infty}^{(N)}(x').\nonumber
\end{equation}
It is notationally convenient to set ${\rm KL}_{\Gamma,w}(c;M,N) = 0$ if $(w,c)$ is not admissible.

The Kloosterman sum for the trivial Weyl element is easy to compute. Indeed one gets
\begin{equation}
	{\rm KL}_{\Gamma,1}(c;M,N) = \delta_{\substack{M=N,\\ c=(1,1)}} \cdot \mathcal{N}(\Gamma).\label{eq:triv_KS}
\end{equation}
The general case is much harder.

Note that the Kloosterman sums are actually local in nature. In particular, for $\Gamma_{\rm pa}(q)$, they factor as
\begin{equation}
	{\rm KL}_{\Gamma_{\rm pa}(q),w}(c;\mathbf{1},\mathbf{1}) = {\rm KL}_{\Gamma_{\rm pa}(q),w}(d;\mathbf{1},N_{c'}){\rm KL}_{\Gamma_0,w}(c';\mathbf{1},N_d), \label{fact_KS}
\end{equation}
where $c_i=d_ic_i'$ with $d_i\mid q^{\infty}$ and $(c_i',q)=1$.\footnote{Given two numbers $a,b\in \mathbb{N}$ we write $a\mid b^{\infty}$ if every prime divisor of $a$ also divides $b$.} Note that the coordinates of $N_{c'}$ are co-prime to $q$ and $(N_{d,1}N_{d,2},c_1'c_2')=1$. This is essentially \cite[(4.5)]{Man}. For more details we refer to \cite{St} and \cite[Proposition~2.4]{Fr} where Kloosterman sums for the lattice $\textrm{SL}_n(\Z)$ are treated, but the arguments directly generalize to our setting.

\subsection{Bounds for Kloosterman sums}

Below we will need several estimates for Kloosterman sums. The upshot of the factorization stated in \eqref{fact_KS} is, that it separates the unramified and the ramified contribution nicely. Note that we always have the trivial bound
\begin{equation}
	\vert {\rm KL}_{\Gamma,w}(c;M,N)\vert \leq \sharp  \mathfrak{X}_{\Gamma}(c^{\ast}w).\nonumber
\end{equation}
For $\Gamma=\Gamma_0$ the Kloosterman sets are sufficiently well understood by \cite{Dr2} and we have the estimate
\begin{equation}
	{\rm KL}_{\Gamma_0,w}(c;\mathbf{1},\mathbf{1}) \leq \sharp \mathfrak{X}_{\Gamma_0}(c^{\ast}w) \ll (c_1c_2)^{1+\epsilon}.\nonumber
\end{equation}
For $\Gamma_{\rm pa}(q)$ estimating $\sharp \mathfrak{X}_{\Gamma_{\rm pa}(q)}(c^{\ast}w)$ is a more delicate matter, which we will look at in Lemma~\ref{para_ram_KS} below. Before doing so we will recall some non-trivial estimates for unramified Kloosterman sums:

\begin{theorem}[\cite{Man0}] \label{unram_KS}
Let $N,c\in \mathbb{N}^2$ and suppose that $(N_1N_2,c_1c_2)=1$. Then we have the bounds
\begin{align}
	{\rm Kl}_{\Gamma_0,s_1s_2s_1}(c;\mathbf{1},N) &\ll c_1^{\frac{5}{3}+\epsilon} \text{ for }c_1=c_2,\nonumber \\
	{\rm Kl}_{\Gamma_0,s_2s_1s_2}(c;\mathbf{1},N) &\ll c_1^{\frac{5}{2}+\epsilon} \text{ for }c_1^2=c_2 \text{ and }\nonumber \\
	{\rm Kl}_{\Gamma_0,s_1s_2s_1s_2}(c;\mathbf{1},N) &\ll c_1^{\frac{1}{2}+\epsilon}c_2^{\frac{3}{4}+\epsilon}(c_1,c_2)^{\frac{1}{2}}.\nonumber 
\end{align}
\end{theorem}
\begin{proof}
This is extracted from \cite[Theorem~1.2]{Man0}. Note that the estimates given in \cite[Theorem~1.2]{Man0} are more general, so that it might be useful to to make some comments. First, note that $n_w(c_1,c_2)$ used in \cite{Man0} reads $n_w(c_1,c_2) = (c_1,c_2)^{\ast}w$ in our notation. Second, we recall that ${\rm Kl}_{\Gamma_0,w}(c;\mathbf{1},N) =  0$ unless $(w,c)$ is $(\mathbf{1},N)$-admissible. Since in this case the bounds are trivially true we only need to treat $(w,c)$ that are $(\mathbf{1},N)$-admissible. Finally, we remark that the assumption $(N_1N_2,c_1c_2)=1$ allows us discard all the gcd's involving entries of $N$ and $c$ simultaneously from the estimates given in \cite[Theorem~1.2]{Man0}. 

The estimates are collected as follows:
\begin{itemize}
	\item The bound for ${\rm Kl}_{\Gamma_0,s_1s_2s_1}(c;\mathbf{1},N)$ follows from the $6$th estimate in \cite[Theorem~1.2]{Man0} with $c_1=c_2$, in which case $c_2\mid c_1^2$ is obviously satisfied and we have $(c_1,c_2)=c_1$.
	\item Similarly we obtain the bound for ${\rm Kl}_{\Gamma_0,s_2s_1s_2}(c;\mathbf{1},N)$ from the $7$th estimate in \cite[Theorem~1.2]{Man0} with $c_1^2=c_2$. Again this is a special case of $c_1^2\mid c_2$ and we also have $(c_1^2,c_2)=c_1^2$.
	\item The bound for the long Weyl element is given in the last estimate of \cite[Theorem~1.2]{Man0}.
\end{itemize}  
This completes the argument.
\end{proof}

Ramified Kloosterman sums were computed in \cite[Section~4]{Man} for congruence subgroups of Siegel type. We adapt these results to the lattices $\Gamma_{\rm pa}(q)$.

\begin{lemma}\label{para_ram_KS}
Let $q>1$. Unless $q\mid a$ we have
\begin{equation}
	{\rm KL}_{\Gamma_{\rm pa}(q),s_1s_2s_1}((a,a),\mathbf{1},N)={\rm KL}_{\Gamma_{\rm pa}(q),s_2s_1s_2}((a,a^2),\mathbf{1},N) =0. \nonumber 
	 \nonumber
\end{equation}
Similarly, 
\begin{equation}
	{\rm KL}_{\Gamma_{\rm pa}(q),s_1s_2s_1s_2}((a,b),\mathbf{1},N) = 0 \nonumber
\end{equation}
unless $q\mid a$ and $a\mid b$. If $q$ is prime we have
\begin{align}
	{\rm KL}_{\Gamma_{\rm pa}(q),s_1s_2s_1}((q,q),\mathbf{1},N) &= q^2 \text{ and } \nonumber \\
	{\rm KL}_{\Gamma_{\rm pa}(q),s_2s_1s_2}((q,q^2),\mathbf{1},N) &= 0. \nonumber
\end{align}
Furthermore, if $q$ is prime, $(N_2,q)=1$ and $k\in \{1,2,3\}$, then
\begin{equation}
	{\rm KL}_{\Gamma_{\rm pa}(q),s_1s_2s_1s_2}((q,q^k),\mathbf{1},N) \ll \tau(q^k)q^{2+\frac{k-1}{2}}.\nonumber
\end{equation}
\end{lemma}
\begin{proof}
We start by considering the Weyl element $w=s_1s_2s_1$ with a corresponding modulus $c=(a,a)$. We compute
\begin{multline}
	xwc^{\ast}x' = \left(\begin{matrix} 1 & x_1 & x_2 & x_3 \\ 0&1&x_3' & x_4\\ 0&0&1&0\\0&0&-x_1&1\end{matrix}\right) \left(\begin{matrix}0&0&-a^{-1} & 0\\ 0&-1&0&0\\ a&0&0&0\\0&0&0&-1 \end{matrix}\right)\left(\begin{matrix} 1 & y_1/a & y_2/a & y_3/a \\ 0&1&y_3/a & 0\\ 0&0&1&0\\0&0&-y_1/a & 1 \end{matrix}\right) \\
	= \left(\begin{matrix} ax_2 & x_2y_1-x_1 & -\frac{1}{a}-\frac{x_1y_3}{a}+x_3'y_2+\frac{x_4y_1}{a}  & x_2y_3-x_3\\ax_3' & x_3'y_1-1 & x_3'y_2-\frac{y_3}{a}+\frac{x_4y_1}{a} & y_3x_3'-x_4 \\  a & y_1 & y_2 &y_3 \\-ax_1 & -x_1y_1 & -x_1y_2+\frac{y_1}{a} & -1-x_1y_3\end{matrix}\right).\label{set_s1s2s1}
\end{multline}
Now we will look at what we gain from the condition
\begin{equation}
	\gamma=xwc^{\ast}x'\in \left[\begin{matrix} \mathbb{Z} & \mathbb{Z} & q^{-1}\mathbb{Z} & \mathbb{Z} \\ q\mathbb{Z} & \mathbb{Z} & \mathbb{Z}& \mathbb{Z}\\ q\mathbb{Z}&q\mathbb{Z}&\mathbb{Z}&q\mathbb{Z}&\\ q\mathbb{Z}&\mathbb{Z}&\mathbb{Z}&\mathbb{Z}\end{matrix} \right] = \Gamma_{\rm pa}(q).\nonumber
\end{equation}
In complete generality we only make the observation that $q\mid a$ is necessary. Now we specialize to $a=q$. First, we look at the third row and deduce that $y_1=y_3 = 0$ and $y_2\in \mathbb{Z}/q\mathbb{Z}$. Studying the first column gives $x_1=x_3'=0$ and $x_2=q\tilde{x}_2$ with $\tilde{x}_2\in \mathbb{Z}/q\mathbb{Z}$. Finally, from the second row we extract $x_4=0$. This is all the information we can get. In summary we have
\begin{equation}
	{\rm KL}_{\Gamma_{\rm pa}(q),s_1s_2s_1}((q,q),\mathbf{1},N) = q^2,\nonumber
\end{equation}
since both $x_2$ and $y_2$ are not seen by the characters.

We turn towards $w=s_2s_1s_2$ and $c=(a,a^2)$. One computes
\begin{multline}
	xwc^{\ast}x' = \left(\begin{matrix} 1 & x_1 & x_2 & x_3 \\ 0&1&x_3' & x_4\\ 0&0&1&0\\0&0&-x_1&1\end{matrix}\right) \left(\begin{matrix}0&0&0&a^{-1} \\ 0& 0&-a^{-1}&0\\ 0&-a&0&0\\a&0&0&0 \end{matrix}\right)\left(\begin{matrix} 1 & 0 & y_1/a^2 & y_2/a^2 \\ 0&1&y_2/a^2 & y_3/a^2\\ 0&0&1&0\\0&0&0 & 1 \end{matrix}\right) \\ 
	=\left(\begin{matrix} ax_3 & -ax_2  & \frac{x_3y_1}{a}-\frac{x_2y_2}{a}-\frac{x_1}{a}& \frac{1}{a}-\frac{x_2y_3}{a}+\frac{x_3y_2}{a} \\ ax_4 & -ax_3' & -\frac{1}{a}-\frac{x_3'y_2}{a}+\frac{x_4y_1}{a} & \frac{x_4y_2}{a}-\frac{x_3'y_3}{a} \\ 0 & -a & -\frac{y_2}{a}& -\frac{y_3}{a} \\ a & ax_1 & \frac{y_1}{a}+\frac{x_1y_2}{a} & \frac{y_2}{a}+\frac{x_1y_3}{a}	\end{matrix}\right).  \label{set_s2s1s2}
\end{multline}
In general we must have $q\mid a$. For the special case $a=q$ we further look at the third row to obtain $y_3=0$ and $y_2=\tilde{y}_2q$ with $\tilde{y}_2\in \mathbb{Z}/q\mathbb{Z}$. The first column gives $x_3=x_3'=\frac{1}{q}\tilde{x}_3$ with $\tilde{x}_3\in \mathbb{Z}/q\mathbb{Z}$  and $x_4=0$. Considering the second column reveals $x_1=\frac{1}{q}\tilde{x}_1$  and $x_2=\frac{1}{q}\tilde{x}_2$ for $\tilde{x}_1,\tilde{x}_2\in \mathbb{Z}/q\mathbb{Z}$. The remaining entries give the congruence conditions
\begin{align}
	\tilde{x}_3\cdot\tilde{y}_2 + 1 & \equiv 0 \text{ mod }q, \nonumber \\
	\tilde{x}_1\cdot \tilde{y}_2 +y_1 &\equiv 0 \text{ mod }q, \nonumber \\
	\tilde{x}_3y_1-\tilde{x}_1-q\tilde{x}_2\tilde{y}_2 &\equiv 0 \text{ mod }q. \label{para_last_cong}
\end{align}
It turns out that the $\tilde{x}_1$ sum is free and by character orthogonality we conclude that
\begin{equation}
	{\rm KL}_{\Gamma_{\rm pa}(q),s_1s_2s_1}((q,q^2),\mathbf{1},N) = 0. \nonumber
\end{equation}

Finally, we turn towards $w=s_1s_2s_1s_2$ and $c=(a,b)$. We compute
\begin{multline}
	xwc^{\ast}x' = \left(\begin{matrix} 1 & x_1 & x_2 & x_3 \\ 0&1&x_3' & x_4\\ 0&0&1&0\\0&0&-x_1&1\end{matrix}\right) \left(\begin{matrix}0&0&-1/a&0 \\ 0& 0&0&-a/b\\ a&0&0&0\\0&b/a&0&0 \end{matrix}\right)\left(\begin{matrix} 1 & y_1/a & y_2/a & y_3/a \\ 0&1&ay_3'/b & ay_4/b\\ 0&0&1&0\\0&0&-y_1/a & 1 \end{matrix}\right) \\ 
	=\left(\begin{matrix} ax_2 & x_2y_1+\frac{bx_3}{a} & -\frac{1}{a}+\frac{x_1y_1}{b}+x_2y_2+x_3y_3' & -\frac{ax_1}{b}+x_2y_3+x_3y_4 \\ ax_3' & x_3'y_1+\frac{bx_4}{a} & \frac{y_1}{b}+x_3'y_2+x_4y_3' & -\frac{a}{b}+x_3'y_3+x_4y_4 \\ a & y_1 & y_2 & y_3 \\ -ax_1 & \frac{b}{a}-x_1y_1 & y_3'-x_1y_2 & y_4-x_1y_3	\end{matrix}\right).  \label{set_s1s2s1s_2}
\end{multline}
The third entry in the first column shows that $q\mid a$. By looking at the third entry of the second column we obtain $y_1\in q\Z$. Since we are working modulo $\Z$ we can put $y_1=0$. The last entry in the second column gives rise to the condition $\frac{b}{a}\in \Z$. This forces $a\mid b$ as claimed. 

Now we specialize to $a=q$ and $b=q^k$, for $k=1,2,3$. Looking at the last two rows and the first two columns we observe that $x_1=x_3=x_3' = y_1 = y_3 = y_3'=0$. Furthermore $y_2\in \mathbb{Z}/q\mathbb{Z}$, $y_4\in \mathbb{Z}/q^{k-1}\mathbb{Z}$ and $x_2=\frac{1}{q}\tilde{x}_2$ with $\tilde{x}_2\in \mathbb{Z}/q\mathbb{Z}$ and $x_4=\frac{a}{b}\tilde{x}_4$ with $\tilde{x}_4\in \mathbb{Z}/q^{k-1}\mathbb{Z}$. Finally, the last entry in the second row reveals
\begin{equation}
	\tilde{x}_4y_4\equiv 1 \text{ mod }q^{k-1}.\nonumber
\end{equation}
Thus we end up with
\begin{equation}
	{\rm KL}_{\Gamma_{\rm pa}(q),s_1s_2s_1s_2}((q,q^k),\mathbf{1},N) = q^2\sum_{x\in (\mathbb{Z}/q^{k-1}\mathbb{Z})^{\times}}e\left(\frac{x+N_2\overline{x}}{q^{k-1}}\right).\nonumber
\end{equation}
We identify the remaining $x$-sum as a classical Kloosterman sum. The desired estimate follows from standard bounds for the latter.
\end{proof}

\begin{remark}
The proof of Lemma~\ref{para_ram_KS} is essentially elementary. The same brute force strategy can be applied to any other lattice $\Gamma$ (e.g. Klingen congruence subgroup, Borel-type congruence subgroup or principal congruence subgroup). Note that we only perform detailed computations for a restricted range of admissible moduli $c=(c_1,c_2)$. It is an interesting problem, which goes beyond the scope of this article, to determine how ramified Kloosterman sums behave for general moduli $c$. 
\end{remark}

\section{Automorphic preliminaries}

We turn towards the automorphic side of the medal. After introducing the basic objects, namely Hecke-Siegel-Maa\ss\  forms, we consider their adelic lifts and the corresponding representations. Finally, we recall Arthur's parametrization of the discrete spectrum and draw some first conclusions.

\subsection{Automorphic forms and representations}

We equip the space $L^2(X_{\Gamma})$ with the inner product
\begin{equation}
	\langle f,g \rangle_{\Gamma} = \int_{X_{\Gamma}}f(x)\overline{g(x)}dx. \nonumber
\end{equation}
D'apr\`es Langlands it is well known that we have the decomposition
\begin{equation}
	L^2(X_{\Gamma}) = L^2_{\rm disc}(X_{\Gamma})\oplus L^2_{\rm Eis}(X_{\Gamma}). \nonumber
\end{equation}
In this note we focus exclusively on the discrete part, which further decomposes in a space of cusp forms and residues of Eisenstein series:
\begin{equation}
	L^2_{\rm disc}(X_{\Gamma}) = L^2_{\rm cusp}(X_{\Gamma})\oplus L^2_{\rm res}(X_{\Gamma}).\nonumber
\end{equation} 
Each element $\varpi\in L^2_{\rm disc}(X_{\Gamma})$, which is an eigenfunction of all invariant differential operators, comes with a spectral parameter $\mu_{\varpi} \in \mathbb{C}^2$. Below in Section~\ref{sec_arth} we will give a finer decomposition of the discrete part based on Arthur's classification of automorphic representations for $G$.

The (classical) Jacquet period of $\varpi\in L^2(X_{\Gamma})$ is given by
\begin{equation}
	\mathcal{W}_{\varpi}(g) = \int_{U(\mathbb{Z})\backslash U(\mathbb{R})}\varpi(ug) \overline{\boldsymbol{\psi}_{\infty}(u)}du.\nonumber
\end{equation}
If this does not vanish identically as a function on $g$, then we call $\varpi$ generic. Now let $\varpi$ be an eigenform with spectral parameter $\mu_{\varpi}$. Then we can write
\begin{equation}
		\mathcal{W}_{\varpi}(u\iota(y)k) = A_{\varpi}(\mathbf{1}) W_{\mu_{\varpi}}(y)\boldsymbol{\psi}_{\infty}(u), \label{class_whitt}
\end{equation}
where $A_{\varpi}(\mathbf{1})$ is the first Fourier coefficient of $\varpi$ and $W_{\mu_{\varpi}}$ is the standard Whittaker function as defined in \cite{Is}.

Adelically one considers functions in $L^2(Z(\A)G(\mathbb{Q})\backslash G(\mathbb{A}))$. Note that we can identify
\begin{equation}
L^2(Z(\mathbb{A})G(\mathbb{Q})\backslash G(\mathbb{A})) = L^2(G'(\mathbb{Q})\backslash G'(\mathbb{A})),\nonumber
\end{equation}
where $G'={\rm PGSp}_4$. The adelic Petersson norm is defined by
\begin{equation}
	\langle \phi,\phi\rangle = \int_{Z(\mathbb{A})G(\mathbb{Q})\backslash G(\mathbb{A})}\vert\phi(g)\vert^2d^{\rm ta}g, \text{ for }\phi\in L^2(Z(\A)G(\mathbb{Q})\backslash G(\mathbb{A})).\nonumber
\end{equation}
Further, we have the global Whittaker period given by
\begin{equation}
	\mathcal{W}_{\phi}^{(\mathbb{A})}(g) =\int_{U(\mathbb{Q})\backslash U(\mathbb{A})}\phi(ug)\overline{\boldsymbol{\psi}(u)}du.\nonumber
\end{equation}

We call $\phi\in L^2(Z(\A)G(\mathbb{Q})\backslash G(\mathbb{A}))$ an adelic lift of $\varpi\in L^2(X_{\Gamma})$ and write $\phi=\phi_{\varpi}$, if 
\begin{equation}
	\varpi(g) = \phi_f(\iota_{\infty}(g)) \text{ for }g\in G_0(\mathbb{R}).\nonumber
\end{equation}
Note that $\phi_{\varpi}$ is necessarily right-$K_{\Gamma}$-invariant. For the paramodular group $\Gamma_{\rm pa}(q)$ of level $q$ the lifting procedure is well defined. For more context we refer to \cite{ASc}, where the case of classical Siegel modular forms is explained in detail. 

By strong approximation we can relate the classical Jacquet period and the global Whittaker period as follows:
\begin{equation}
	\mathcal{W}_{\phi}^{(\mathbb{A})}(\iota_{\infty}(g)) =  \mathcal{W}_{\varpi}(g), \label{comp_whitt}
\end{equation}
for $\phi=\phi_{\varpi}$ and $g\in G_0(\mathbb{R})$. Similarly, one gets
\begin{equation}
	\langle \phi_{\varpi},\phi_{\varpi}\rangle =  2\frac{\langle \varpi,\varpi\rangle_{\Gamma}}{\mathcal{V}(\Gamma)}.\label{eq:comp_innerprod}
\end{equation}

It will be convenient to package an orthogonal basis of $L^2_{\rm dis}(X_{\Gamma})$ into pieces associated to (irreducible) automorphic representations $\pi$. We write $\pi\mid X_{\Gamma}$ if $\pi$ has non-trivial $K_{\Gamma}$ invariant elements. The finite dimensional space generated by these elements will be denoted by $\pi^{K_{\Gamma}}$. If $\pi$ contributes to the discrete spectrum of $L^2(X_{\Gamma})$, we put
\begin{equation}
	V_{\Gamma}(\pi) = \{ g\mapsto \pi(\iota_{\infty}(g))v\colon v\in \pi^{K_{\Gamma}} \}\subseteq L^2_{\rm disc}(X_{\Gamma}).\label{def:space_of_pi}
\end{equation}
We write $\mathcal{O}_{\Gamma}(\pi)$ for an orthogonal basis of $V_{\Gamma}(\pi)$.

By Flath's theorem we can factor $\pi=\pi_{\infty} \otimes \bigotimes_p \pi_p$ into local representations. In complete analogy to the global case we write $\pi_p^{K_{\Gamma,p}}$ for the space generated by right-$K_{\Gamma,p}$-invariant elements of $\pi_p$ and we let $\mathcal{O}_{K_{\Gamma,p}}(\pi_p)$ be an orthogonal basis of $\pi_p^{K_{\Gamma,p}}$. Note that if $p\nmid D_{\Gamma}$, then $\pi_p^{K_{\Gamma,p}}$ is generated by the (up to scaling) unique spherical element $v_p^{\circ}$. Thus, globally
\begin{equation}
	\{v_{\infty}^{\circ}\otimes \bigotimes_{p\nmid D_{\Gamma}}v_p^{\circ}\otimes\bigotimes_{p\mid D_{\Gamma}}v_p\colon v_p\in \mathcal{O}_{K_{\Gamma,p}}(\pi_p) \text{ for }p\mid D_{\Gamma}\} \nonumber
\end{equation}
can be taken as an orthogonal basis of $\pi^{K_{\Gamma}}$. We will always choose $\mathcal{O}_{\Gamma}(\pi)$ by descending the above basis of $\pi^{K_{\Gamma}}$ to $V_{\Gamma}(\pi)$.

Note that each $\varpi\in V_{\Gamma}(\pi)$ is an eigenfunction of all invariant differential operators (and even of all but finitely many Hecke-operators). The spectral parameter of $\varpi$ only depends on $\pi_{\infty}$. Thus, we can write 
\begin{equation}
	\mu_{\varpi} = \mu(\pi_{\infty}).\nonumber
\end{equation}
Indeed since $\pi_{\infty}$ is spherical it can be realized as Langlands quotient of an induced representation 
\begin{equation}
	\pi_{\infty} = L(\chi_1\times \chi_2\rtimes \sigma).\nonumber
\end{equation} 
Note that by assumption on the central character $\chi$ we have $\chi_1\chi_2\sigma^2 = \chi_{\infty}\in \{ \mathbf{1},{\rm sgn}\}$. We write $\chi_i={\rm sgn}^{\rho_i}\vert \cdot \vert^{\alpha_i}$ and $\sigma = {\rm sgn}^{\rho'}\vert \cdot\vert^{\beta}$. Note that we have $\alpha_1+\alpha_2+2\beta = 0$. We associate the spectral parameter 
\begin{equation}
	\mu(\pi_{\infty}) = (\frac{\alpha_1+\alpha_2}{2},\frac{\alpha_1-\alpha_2}{2}) \in \mathbb{C}^2.\nonumber
\end{equation}
Note that the transfer of $\pi_{\infty}$ to ${\rm GL}_4$ is self-dual and has spectral parameter
\begin{equation}
	(\frac{\alpha_1+\alpha_2}{2},\frac{\alpha_1-\alpha_2}{2},\frac{-\alpha_1+\alpha_2}{2},\frac{-\alpha_1-\alpha_2}{2}) \in \mathbb{C}^4.\nonumber
\end{equation}
With this notation at hand we can express
\begin{equation}
N_{\Gamma}(\sigma;M) = \sum_{\substack{\pi\mid X_{\Gamma}, \\ \Vert \mu(\pi_{\infty})\Vert \leq M,\\ \sigma(\pi_{\infty})\geq \sigma}}\dim_{\mathbb{C}}(V_{\Gamma}(\pi)).\nonumber
\end{equation}

Recall that $\pi$ is called (globally) generic if the global Whittaker period $\mathcal{W}^{(\mathbb{A})}_{\phi}$ is non-zero (as function on $G(\mathbb{A})$) for some element $\phi$ of $\pi$. On the other hand we call $\pi$ abstractly generic (or locally everywhere generic) if $\pi_{\infty}$ and $\pi_p$ are generic for all $p$ in the sense that they admit a non-trivial Whittaker functional. According to \cite[Proposition 8.3.2]{Ar2} $\pi$ is generic if and only if $\pi$ is abstractly generic.\footnote{The implication globally generic implies everywhere locally generic is of course essentially trivial. The other direction however is highly non-trivial. As an alternative to Arthur's seminal work one can argue as in \cite[(6.8)]{Sh} to establish it.}

To each generic representation $\pi_p$ we fix a Whittaker functional leading to an isomorphism $v\mapsto W_v$ between $\pi_p$ and its Whittaker model $\mathcal{W}(\pi_p,\boldsymbol{\psi}_p)$. If $\pi_p$ is unramified we fix (once and for all) the normalization such that
\begin{equation}
	W_{v_{p}^{\circ}}(1)=1.\nonumber
\end{equation}
We want to make similar accommodations at the  archimedean place. To do so we note that if $\pi_{\infty}$ is generic and spherical, then $\pi_{\infty}$ is a principal series. This is \cite[Proposition~3.4.1]{Co}. In particular, we have
\begin{equation}
	\pi_{\infty}\vert_{G_0(\mathbb{R})} = {\rm Ind}_{B_0(\mathbb{R})}^{G_0(\mathbb{R})}(\vert \cdot \vert^{\mu_1}\boxtimes \vert\cdot \vert^{\mu_2}), \nonumber
\end{equation}
where $\mu(\pi_{\infty}) = (\mu_1,\mu_2)$. This is the setting also described in \cite{CI}. We normalize the Whittaker functional so that
\begin{equation}
	W_{v_{\infty}^{\circ}}(1) = W_{\mu(\pi_{\infty})}(1). \nonumber
\end{equation} 
If $\varpi\in \mathcal{O}_{\Gamma}(\pi)$ corresponds to $v_{\infty}^{\circ}\otimes \bigotimes_{p\nmid D_{\Gamma}}v_p^{\circ}\otimes\bigotimes_{p\mid D_{\Gamma}}v_p$, then by uniqueness of the Whittaker functional we must have
\begin{equation}
	\frac{\vert A_{\varpi}(1) W_{\mu_{\varpi}}(1)\vert^2}{\langle\varpi,\varpi\rangle} =  \frac{\vert\mathcal{W}_{\varpi}(1)\vert^2}{\langle\varpi,\varpi\rangle} =  \frac{\vert\mathcal{W}_{\phi_w}^{(\mathbb{A})}(1)\vert^2}{\mathcal{V}(\Gamma)\cdot \langle\phi_{\varpi},\phi_{\varpi}\rangle} =\frac{C(\pi)}{\mathcal{V}(\Gamma)}\cdot \frac{\vert W_{\mu(\pi_{\infty})}(1)\vert^2\cdot \prod_{p\mid D_{\Gamma}}\vert W_{v_p}(1)\vert^2}{\langle\phi_{\varpi},\phi_{\varpi}\rangle}.\nonumber
\end{equation}
The constant (as well as the normalization of the Whittaker functionals at the places $p\mid D_{\Gamma}$) is usually determined via the Rankin-Selberg method. Unfortunately there are some hurdles when carrying this out in generally. These can be resolved under assumption of a conjecture from Lapid and Mao given in \cite{LM}. Here we give a formulation of this conjecture along the lines of \cite[Section~4]{PSS} adapted to our notation: 

\begin{conjecture}[Lapid-Mao] \label{conj}
Let $\pi=\pi_{\infty}\otimes \bigotimes_p \pi_p$ be an irreducible, unitary, generic cuspidal automorphic representation with trivial central character. Suppose that $\pi\mid X_{\Gamma}$. Then, for $\varpi\in \mathcal{O}_{\Gamma}(\pi)$ we have
\begin{equation}
	\frac{\vert A_{\varpi}(1)\vert^2}{\langle \varpi,\varpi\rangle} = \frac{2^{5-c}}{\mathcal{V}(\Gamma)}\cdot \frac{\Lambda^{D_{\Gamma}}(2)\Lambda^{D_{\Gamma}}(4)}{\Lambda^{D_{\Gamma}}(1,\pi,{\rm Ad})}\cdot \prod_{p\mid D_{\Gamma}}J_{\pi_p}(v_p), \nonumber
\end{equation}
where $c=1$ if $\pi$ is stable (i.e. of general type) and $c=2$ if $\pi$ is endoscopic.	Furthermore, $\Lambda^{D_{\Gamma}}(s,\pi,{\rm Ad})$ is the completed adjoint $L$-function of $\pi$ with Euler factors at primes dividing $D_{\Gamma}$ omitted. (Similarly $\Lambda^{D_{\Gamma}}$ is the completed Riemann-Zeta function with Euler factors at $p\mid D_{\Gamma}$ omitted.) Finally, the local period $J_{\pi_p}(v_p)$ is defined by
\begin{equation}
	J_{\pi_p}(v_p) = \int_{U(\mathbb{Q}_p)}^{\rm st}\frac{\langle \pi_p(u)v_p,v_p\rangle}{\langle v_p,v_p\rangle} \overline{\boldsymbol{\psi}_p(u)}du. \nonumber
\end{equation}
\end{conjecture}

A special case of Conjecture~\ref{conj} was solved in \cite{CI}. Since their result is important for the unconditional part of Theorem~\ref{main_theorem} we recall it here:
\begin{theorem}[Theorem~2.1, \cite{CI}]\label{CI_th}
Let $q\geq 1$ be square-free, then we have
\begin{equation}
	\frac{\vert A_{\varpi}(\mathbf{1})\vert^2}{\langle \varpi,\varpi\rangle_{\Gamma_{\rm pa}(q)}} = \frac{2^{5-c}}{\mathcal{V}(\Gamma_{\rm pa}(q))}\cdot \frac{\Lambda^{q}(2)\Lambda^{q}(4)}{\Lambda^{q}(1,\pi,{\rm Ad})}\cdot \prod_{p\mid q}\frac{p}{(1-p^{-2})^2}, \nonumber
\end{equation}
for $\varpi\in \mathcal{O}_{\Gamma_{\rm pa}(q)}(\pi)$ and $\pi\mid X_{\Gamma_{\rm pa}(q)}$ generic.
\end{theorem}
\begin{proof}
To see how this follows from \cite[Theorem~2.1]{CI} we first combine \eqref{class_whitt} with \eqref{comp_whitt} to see that
\begin{equation}
	\mathcal{W}_{\phi_{\varpi}}^{(\A)}(1) = \mathcal{W}_{\varpi}(1) = A_{\varpi}(\mathbf{1})W_{\mu_{\varpi}}(1).\nonumber
\end{equation}
Note that our $\mathcal{W}_{\phi_{\varpi}}^{(\A)}$ is simply denoted by $W$ in \cite[Section~2]{CI}. We also recall \eqref{eq:comp_innerprod} and note that $\langle\phi_{\varpi},\phi_{\varpi}\rangle$ is normalized as in \cite[Section~2]{CI}. 

Next we observe that since $\varpi\in \mathcal{O}_{\Gamma}(\pi)$ our form $\phi_{\varpi}$ has locally the transformation behavior described in \cite[(2.4)]{CI}. In particular, we can use \cite[Theorem~2.1]{CI} to compute the Petersson norm of $\varpi$. Indeed, using notation from \cite{CI}, we obtain
\begin{multline}
	\frac{\vert A_{\varpi}(\mathbf{1})W_{\mu_{\varpi}}(1)\vert^2}{\langle \varpi,\varpi\rangle_{\Gamma_{\rm pa}(q))}}= \frac{2}{\mathcal{V}(\Gamma_{\rm pa}(q))}\cdot \frac{\vert \mathcal{W}_{\phi_{\varpi}}^{(\A)}(1)\vert^ 2}{\langle \phi_{\varpi},\phi_{\varpi}\rangle} \\ = \frac{2}{\mathcal{V}(\Gamma_{\rm pa}(q))}\cdot\vert W_{\mu_{\varpi}}(1) \vert^2\cdot \left(2^c\frac{L(1,\pi,\textrm{Ad})}{\Delta_{\textrm{PGSP}_4}}\prod_vC(\pi_v)\right)^{-1} \label{eq:not_det}
\end{multline}
Taking the quotient makes the expression scaling invariant, so that there is no problem in using \cite[(2.5) and (2.6)]{CI} to compute it. We are done after dividing both sides by $\vert W_{\mu_{\varpi}}(1)\vert^2$ and explicating the right hand side of \eqref{eq:not_det} using the exact values for $C(\pi_v)$ given in the statement of \cite[Theorem~2.1]{CI}.
\end{proof}

For more general $\Gamma$ estimating the first Fourier coefficient from below is a hard problem. By assuming Conjecture~\ref{conj} it reduces to a purely local computation, which is still hard. We are however tempted to conjecture that
\begin{equation}
	\prod_{p\mid D_{\Gamma}}\sum_{v_p\in \mathcal{O}_{K_{\Gamma,p}}(\pi_p)} J_{\pi_p}(v_p) \asymp \dim V_{\pi}(\Gamma)\cdot \mathcal{N}(\Gamma), \label{con_2}
\end{equation}
for generic $\pi=\otimes_v \pi_v$. Of course this is speculation but we have the following \textit{evidence}:
\begin{itemize}
	\item For $\Gamma=\Gamma_{\rm pa}(q)$ with square-free $q$ this follows from the main result of \cite{CI}. Furthermore, if $\pi_p$ is a simple supercuspidal representation (this forces $p^5\mid q$), the local integral $J_{\pi_p}(v_p)$ has recently been computed for the local new-vector $v_p$ in \cite{PSS}.
	
	\item For congruence subgroups associated to parabolic subgroups (i.e Siegel congruence subgroups, Klingen congruence subgroups or Borel-type congruence subgroups) of square-free level local computations in the spirit of \cite{DPSS} lead to the desired result.
	
	\item For the principal congruence subgroup of square-free level the arguments from \cite{AB} can be adapted to prove \eqref{con_2} in the setting at hand. Note that carrying this out requires a little bit of care, since not all depth zero supercuspidal representations are generic.
\end{itemize}

\subsection{The parametrization of the discrete spectrum}\label{sec_arth}

The parametrization of the discrete spectrum of ${\rm GSp}_4$ in terms of discrete automorphic forms of ${\rm GL}_n$ has been announced in \cite{Ar1}. Building on the seminal monograph \cite{Ar2} this has now been established in \cite{GT}. This parametrization will be important for us in order to upgrade the density estimate for spherical, generic representations to a density theorem for the full (spherical) discrete spectrum. This will require a good understanding of the local constituents of the Arthur packets. For trivial central character the representations factor through $G'={\rm PGSp}_4$ and an explicit parametrization was worked out in \cite{Sc1, Sc2}. We will follow the notation within these references.

Using the exceptional isomorphism ${\rm SO}_5\cong G'$ we can express the central result from \cite{Ar1} as
\begin{equation}
	L^2_{\rm disc}(G'(\mathbb{Q})\backslash G'(\mathbb{A})) \equiv \bigoplus_{\psi\in \boldsymbol{\psi}_2(G')}\bigoplus_{\pi\in \Pi_{\psi}\colon \langle \cdot,\pi\rangle=\epsilon_{\psi}}\pi. \label{para_disc}
\end{equation}
See also \cite[(1.2)]{Sc1}. We will introduce the missing notation on the way of gathering all the necessary information contained in \eqref{para_disc}. In the following we will classify the parameters $\psi$ by type. Relevant types are \textbf{(G)}, \textbf{(Y)}, \textbf{(Q)}, \textbf{(P)}, \textbf{(B)} and \textbf{(F)}. Accordingly we decompose
\begin{equation}
	L^2_{\rm disc}(G'(\mathbb{Q})\backslash G'(\mathbb{A})) = \bigoplus_{\ast \in \{ \mathbf{G}, \mathbf{Y}, \mathbf{Q}, \mathbf{P}, \mathbf{B}, \mathbf{F}\}}L^2_{\mathbf{(}\boldsymbol{\ast}\textbf{)}}(G'(\mathbb{Q})\backslash G'(\mathbb{A})).\nonumber
\end{equation}
We will now spend the rest of this section describing the Arthur parameters of each type in detail.

First, the set $\boldsymbol{\psi}_2(G')$ consists of formal expressions
\begin{equation}
	\psi = (\mu_i\boxtimes \nu_1)\boxplus \ldots\boxplus (\mu_r\boxtimes \nu_r).\nonumber
\end{equation}
where $\mu_i$ is a self-dual, cuspidal automorphic representation of ${\rm GL}_{m_i}(\mathbb{A})$ and $\nu_i$ is the irreducible representation of ${\rm SL}_2(\mathbb{C})$ of dimension $n_i$. Furthermore, the following conditions need to be satisfied:
\begin{enumerate}
	\item $m_1n_1+\ldots+m_rn_r=4$;
	\item $\mu_i\boxtimes \nu_i \neq \mu_j\boxtimes \nu_j$ for $i\neq j$;
	\item If $n_i$ is odd (resp. even) then $\mu_i$ is symplectic (resp. orthogonal).
\end{enumerate}
Analyzing these conditions as in \cite{Sc1} leads to different types of parameters. They are as follows:
\begin{enumerate}
	\item \textbf{General Type (G):} $\psi = \mu\boxtimes 1$ for a self-dual symplectic (i.e. $L(s,\mu,\wedge^2)$ has a pole at $s=1$) unitary cuspidal automorphic representation of ${\rm GL}_4(\mathbb{A})$.
	\item \textbf{Yoshida Type (Y):} $\psi=(\mu_1\boxtimes 1)\boxplus (\mu_2\boxtimes 1)$ for distinct unitary cuspidal automorphic representations of ${\rm GL}_2(\mathbb{A})$ with trivial central character.
	\item \textbf{Soundry Type (Q):} $\psi=\mu\boxtimes \nu(2)$ for a self-dual unitary cuspidal automorphic representation $\mu$ of ${\rm GL}_2(\mathbb{A})$ with non-trivial central character $\omega_{\mu}$. ($\nu(2)$ is the two dimensional irreducible representation of ${\rm SL}_2(\mathbb{C})$.)
	\item \textbf{Saito-Kurokawa Type (P):} $\psi=(\mu\boxtimes 1)\boxplus (\sigma\boxtimes \nu(2))$ for a unitary cuspidal automorphic representation $\mu$ of ${\rm GL}_2(\mathbb{A})$ with trivial central character and a quadratic Hecke character $\sigma$.
	\item \textbf{Howe-Piatetski-Shapiro Type (B):} $\psi=(\chi_1\boxtimes\nu(2))\boxplus(\chi_2\boxtimes \nu(2))$ for two distinct quadratic Hecke characters.
	\item \textbf{Finite Type (F):} $\psi=\xi\boxtimes \nu(4)$ For a quadratic Hecke character $\xi$. ($\nu(4)$ is the four dimensional irreducible representation of ${\rm SL}_2(\mathbb{C})$.)
\end{enumerate}

To each packet we attach the group $\mathcal{S}_{\psi}$. For all practical purposes it is sufficient to know that $\mathcal{S}_{\psi} = 1$ if $\psi$ is of type \textbf{(G)}, \textbf{(Q)} or \textbf{(F)}. For the remaining types, namely \textbf{(Y)}, \textbf{(P)} and \textbf{(B)} we have $\mathcal{S}_{\psi}=\{\pm 1\}$.

The global parameter $\psi$ has a localization at each place $v\in \{\infty,2,3,\ldots\}$. Formally these are maps 
\begin{equation}
	\psi_v\colon L_{\mathbb{Q}_v}\times{\rm SU}_2 \to {\rm Sp}_4(\mathbb{C}),\nonumber
\end{equation}
satisfying certain compatibility conditions. In our setting they can be given quite explicitly in terms of the local Langlands parameters of the ${\rm GL}_n$ objects that appear in the corresponding global parameter. We omit their explicit description and refer to \cite[Section~1.2]{Sc1} for details. As in the global case each local parameter comes with a local centralizer group $\mathcal{S}_{\psi_v}$ and we have canonical maps
\begin{equation}
	\mathcal{S}_{\psi}\to \mathcal{S}_{\psi_v}.\nonumber
\end{equation}

Now we attach to each local parameter a finite packet of admissible representations $\Pi_{\psi_v}$ as in \cite[Theorem~1.5.1]{Ar2}. Each packet comes with a canonical map
\begin{equation}
	\Pi_{\psi_v}\ni \pi_v \mapsto \langle \cdot,\pi_v\rangle\in \widehat{\mathcal{S}}_{\psi_v}. \nonumber
\end{equation}
Note that if $\pi_v$ is unramified we have $\langle \cdot ,\pi_v\rangle=1$. The global packet is then given by
\begin{equation}
\Pi_{\psi} = \{\pi=\oplus_v\pi_v\colon \pi_v\in \Pi_{\psi_v}\}.\nonumber
\end{equation}
This is all we want to say about the parameters $\psi$ in general. We now turn towards the individual analysis of packets for the different types.

\subsubsection{Parameters of type (G) and (Y)} For a parameter $\psi$ of type \textbf{(G)} or \textbf{(Y)} the local parameter $\psi_v$ is trivial on ${\rm SU}_2$. We can therefore interpret it as a classical $L$-parameter and it turns out that the local packet $\Pi_{\psi_v}$ coincides with the packets defined by the local Langlands correspondence. See \cite{GaT} for the latter in the case of ${\rm GSp}_4$. In particular, the elements of $\Pi_{\psi_v}$ are irreducible and unitary. 

\begin{lemma}\label{tempered_pack}
Let $\psi$ be an parameter of type \textbf{(G)} or \textbf{(Y)}. Then $\Pi_{\psi}$ contains precisely one generic member $\pi_{\psi}^{\rm gen}$. For $q>1$ arbitrary we have
\begin{equation}
	\sum_{\pi\in \Pi_{\psi}\colon \langle \cdot ,\pi\rangle=\epsilon_{\psi}} \dim V_{\Gamma_{\rm pa}(q)}(\pi) = \dim V_{\Gamma_{\rm pa}(q)}(\pi_{\psi}^{\rm gen}).\label{eq:in_proof}
\end{equation}
\end{lemma}
\begin{proof}
That $\Pi_{\psi}$ contains precisely one generic member follows form \cite[Proposition~8.3.2]{Ar2}. (See also \cite[Theorem~1.1]{Sc1}.) Note that each $L$-packet for ${\rm GSp}_4$ has at most two elements. Of course it suffices to check those $L$-packets with exactly two elements. These are
\begin{equation}
	\{{\rm VIa}, {\rm VIb}\},\ \{{\rm VIIIa},{\rm VIIIb}\},\, \{{\rm Va},{\rm Va}^{\ast}\} \text{ and }\{{\rm XIa}, {\rm XIa}^{\ast}\}.\label{L-pack}
\end{equation}
The desired equality follows if we can show that 
\begin{equation}
	\dim_{\mathbb{C}}(\pi')^{K_{\Gamma_{\rm pa}(q),p}} =0,\label{eq:vanishing_para}
\end{equation}
where $\Pi_{\psi_p} = \{\pi_{\psi_p}^{gen},\pi'\}$ is one of the $L$-packets given in \eqref{L-pack}. For the paramodular group $\Gamma_{\rm pa}(q)$ and $p\mid q$ this follows from \cite[Theorem~1.1]{Sc1}. A related argument can be found in the proof of \cite[Lemma~2.5]{Sc1}.
\end{proof}

\begin{remark}
The exact equality \eqref{eq:in_proof} is special for the paramodular group. Indeed, for more general $\Gamma$, the vanishing result \eqref{eq:vanishing_para} will fail. As a consequence the global Arthur packets of  type \textbf{(G)} and \textbf{(Y)} will have non-generic members that contribute to the cuspidal spectrum of $X_{\Gamma}$. However, we expect the estimate
\begin{equation}
	\sum_{\pi\in \Pi_{\psi}\colon \langle \cdot ,\pi\rangle=\epsilon_{\psi}} \dim V_{\Gamma}(\pi) \ll 2^{\omega(D_{\Gamma})}\cdot \dim V_{\Gamma}(\pi_{\psi}^{\rm gen}),\label{eq:expected_temp}
\end{equation}
to hold in great generality. For example, by using \cite[Table~A.15]{RS}, one can easily verify this for Siegel congruence subgroups, Klingen congruence subgroups and Borel-type congruence subgroups. Furthermore, a similar estimate can be given for the principal congruence subgroup. The corresponding local dimensions have been computed in \cite{Br}.
\end{remark}

\subsubsection{Parameters of type (B) and (Q)} It turns out that packets of these types do not contribute to the spectrum of $X_{\Gamma_{\textrm{pa}}(q)}$. This is due to \cite[Proposition~5.1]{Sc2}. We record this fact formally in form of the following Lemma:

\begin{lemma}\label{lm:pack_B_and_Q}
Let $\psi$ be an Arthur parameter of type \textbf{(B)} or \textbf{(Q)} and let $q>1$ be arbitrary. We have $V_{\Gamma_{\rm pa}(q)}(\pi)=\{0\}$ for all $\pi\in \Pi_{\psi}$. 
\end{lemma}
\begin{proof}
Suppose there is $\pi\in \Pi_{\psi}$ with $V_{\Gamma_{\rm pa}(q)}(\pi)\neq \{0\}$. Then $\pi$ would be paramodular at every finite place, which is impossible by \cite[Proposition~5.1]{Sc2}.
\end{proof}

\begin{remark}\label{rem:cap_pack}
One should not expect a vanishing result as in Lemma~\ref{lm:pack_B_and_Q} to hold for arbitrary families of congruence lattices. For example, if $\Gamma$ is a principal congruence subgroups, then packets of both types will contribute to the discrete spectrum of $X_{\Gamma}$. It should be noted that in this case one encounters representations with non-trivial central character. In particular, one has to use \cite{GT} and the explicit description of the local packets can not be directly obtained from \cite{Sc2}. For the types \textbf{(B)} and \textbf{(P)} this can be resolved using \cite{P-S}, but for packages of type \textbf{(Q)} a little more work seems to be necessary.\footnote{I would like to thank R. Schmidt for pointing me to the result in \cite{P-S} and for other very useful tips concerning the (local) Arthur packets in the presence of a non-trivial central character.}
\end{remark}

\subsubsection{Packets of type (P)} Fix a parameter $\psi=(\mu\boxtimes 1)\boxplus (\sigma\boxtimes\nu(2))$ of type \textbf{(P)} given by a unitary cuspidal automorphic representation $\mu$ with trivial central character and a quadratic Hecke character $\sigma$. The global base point $\pi_{\psi}^+$, as defined below \cite[(12)]{Sc2}, is the isobaric constituent of the globally induced representation 
\begin{equation}
	\vert \cdot\vert^{\frac{1}{2}}\sigma\mu\ltimes \vert \cdot \vert^{-\frac{1}{2}}\sigma. \nonumber 
\end{equation}
Note that here $\epsilon_{\psi} = \epsilon(\frac{1}{2},\sigma^{-1}\otimes\mu)$, so that we have the global compatibility condition
\begin{equation}
	\prod_{v}\epsilon(\pi_v)  =\langle -1,\pi\rangle = \epsilon(\frac{1}{2},\sigma^{-1}\otimes\mu).\nonumber
\end{equation}
The local packets depend on the factorization $\mu=\otimes_v\mu_v$. We briefly summarize \cite[Table~2]{Sc2} as follows:
\begin{center}
	{\tabulinesep=1.2mm
	\begin{tabu}{ | c | c | c | c | c|  }
		\hline
		$\mu_v$ & $\chi\times\chi^{-1}$ & $\chi\cdot {\rm St}_{{\rm GL}_2}$ & $\sigma_v\cdot {\rm St}_{{\rm GL}_2}$ & supercuspidal \\ \hline
		$\Pi_{\psi_v}$ & $\{{\rm IIb}\}$ &  $\{{\rm Vb}, {\rm Va}^{\ast}\}$ &  $\{{\rm VIc}, {\rm VIb}\}$ & $\{{\rm XIb}, {\rm XIa}^{\ast} \}$ \\	\hline
	\end{tabu}}
\end{center}
(Note that we have written  the local base point as first entry in each packet.) 

A finer investigation of the properties of packets of type $\mathbf{(P)}$ will naturally lead us to related properties of the underlying $\textrm{GL}_2$ representation $\mu$. In this context we write $\Gamma_H(q)$ for the standard Hecke congruence subgroup of level $q$. In analogy with \eqref{def:space_of_pi} we can define the space $V_{\Gamma_H(q)}(\mu)$ in the $\textrm{GL}_2$-setting. (Here the classical upper half plane $\mathbb{H}$ takes over the role of $\mathbb{H}_2$.) The dimension of $V_{\Gamma_H(q)}(\mu)$ is well understood by the newform theory of Atkin and Lehner. It depends only on the local conductor exponent of $\mu_v$, which we denote by $a(\mu_v)$. The (global) conductor of $\mu$ is given by $c(\mu)=\prod_pp^{a(\mu_p)}$. We are now ready to prove the following result:

\begin{lemma}\label{lm:pack_P}
Let $\psi=(\mu\boxtimes 1)\boxplus (\sigma\boxtimes\nu(2))$ ba an Arthur parameter of type \textbf{(P)}. For $q>1$ arbitrary we have
\begin{equation}
	\sum_{\substack{\pi\in \Pi_{\psi}\\ \langle -1,\pi\rangle =\epsilon(\frac{1}{2},\mu)}} \dim V_{\Gamma_{\rm pa}(q)}(\pi) = \delta_{\epsilon(\frac{1}{2},\mu)=1} \dim V_{\Gamma_{\rm pa}(q)}(\pi_{\psi}^{+}). \nonumber
\end{equation}
Furthermore, we have $\dim V_{\Gamma_{\rm pa}(q)}(\pi_{\psi}^{+})\leq \dim V_{\Gamma_H(q)}(\mu)\ll q^{\epsilon}$ and $V_{\Gamma_{\rm pa}(q)}(\pi_{\psi}^{+}) = \{0\}$ unless the conductor of $\mu$ divides $q$ and $\sigma$ is trivial. Finally, if $\varpi \in V_{\Gamma_{\rm pa}(q)}(\pi_{\psi}^{+})$ is non-zero, then $\sigma_{\varpi}=\frac{1}{2}$. 
\end{lemma}
\begin{proof}
We can follow the proof of \cite[Proposition~5.2]{Sc2} very closely. The main difference for us is that, since $\pi\mid X_{\Gamma_{\rm pa}(q)}$, the representation $\pi_{\infty}$ must be spherical. In particular, we automatically have $\pi_{\infty}=(\pi_{\psi}^+)_{\infty}$. More precisely, if $\pi \in \Pi_{\psi}$  is spherical, then $\mu_{\infty} = \chi\times \chi^{-1}$ (i.e. $\mu$ is of Maa\ss\  type). Looking at \cite[Table~2]{Sc2} we find that
\begin{equation}
	\pi_{\infty} = \chi\sigma 1_{{\rm GL}_2}\ltimes \chi^{-1}. \nonumber
\end{equation}
The associated real part of the spectral-parameter is $\Re \mu(\pi_{\infty}) = (\beta,\frac{1}{2})$, where $\vert\chi(x)\vert=\vert x \vert^{\beta}.$ Note that $\beta = \sigma(\mu_{\infty})$ and $\beta < \frac{1}{2}$. In particular we have $\sigma(\pi_{\infty}) = \frac{1}{2}$.

We continue as in the \cite{Sc2} by noting that, if there is a finite place $v$ where $\sigma$ is ramified, then no representation in $\Pi_{\psi_v}$ has vectors fixed by the paramodular subgroup. Thus, we can assume that $\sigma$ is trivial. In this case we see that only the local base point at finite places has vectors fixed by the paramodular group. In other words, $\dim V_{\Gamma_{\rm pa}(q)}(\pi) \neq 0$ implies that $\pi=\pi_{\psi}^+$. But $\langle -1,\pi_{\psi}^+\rangle=1$, so that the parity condition forces $\epsilon(\frac{1}{2},\mu) = 1$. Looking closer at \cite[Table~2]{Sc2} we find
\begin{itemize}
	\item If $\mu_v=\chi\times \chi^{-1}$, then $(\pi_{\psi}^+)_{v}$ is of type $\textrm{IIb}$. More precisely, in the notation of \cite{RS} we have $(\pi_{\psi}^+)_{v} = \chi \sigma_v 1_{\textrm{GL}_2}\ltimes \chi^{-1}$. Recall that $\sigma_v$ is unramified, so that $a(\sigma_v\chi)=a(\chi)$. Finally, note that $a(\mu_v) = 2a(\chi)$. In this case the $\sigma$ in \cite[Table~A.12]{RS} corresponds to our $\chi^{-1}$.
	\item If $\mu_v=\chi\textrm{St}_{\textrm{GL}_2}$, for $\chi\neq \sigma_v$ quadratic, then $(\pi_{\psi}^+)_{v}=L(\nu^{\frac{1}{2}}\chi\sigma_v\textrm{St}_{\textrm{GL}_2}, \nu^{-\frac{1}{2}}\sigma_v)$ is of type $\textrm{Vb}$. Here we have $a(\mu_v) = 2$. In this case the $\xi$ in \cite[Table~A.12]{RS} corresponds to our $\chi \sigma_v$.
	\item If $\mu_v=\sigma_v\textrm{St}_{\textrm{GL}_2}$, then $(\pi_{\psi}^+)_{v}=L(\nu^{\frac{1}{2}}\textrm{St}_{\textrm{GL}_2}, \nu^{-\frac{1}{2}}\sigma_v)$ is of type $\textrm{VIc}$. Since $\sigma_v$ is unramified, one has $a(\mu_v)=1$. In this case the $\sigma$ in \cite[Table~A.12]{RS} corresponds to our $\sigma_v$.
	\item If $\mu_v$ is supercuspidal then $(\pi_{\psi}^+)_{v}=L(\nu^{\frac{1}{2}}\sigma_v\nu_v, \nu^{-\frac{1}{2}}\sigma_v)$ is of type $\textrm{XIb}$. By unramifiedness of $\sigma_v$ we have $a(\sigma_v\mu_v)=a(\mu_v)$. In this case $\pi$ (resp. $\sigma$) in \cite[Table~A.12]{RS} corresponds to our $\sigma_v\mu_v$ (resp. $\sigma_v$).
\end{itemize}
Using these observations one can study \cite[Table~A.12]{RS} to obtain $$\dim (\pi_{\psi_v}^+)^{K_{\Gamma_{\rm pa}(q),v}} = \lfloor\frac{v(q)-a(\mu_v)}{2}\rfloor.$$ This completes the proof.
\end{proof}

\begin{remark}
A Similar result can be established for other lattice families as well. The most important ingredients to generalize the argument are the corresponding formulae for the dimensions of local fixed vectors. For example, if $\Gamma$ is a congruence subgroup associated to a parabolic (e.g. Siegel congruence subgroup, Klingen congruence subgroup or Borel-type congruence subgroup) with square-free level, then these can be found in \cite[Table~A.15]{RS}. For principal congruence subgroup the relevant local dimension formulae are given in \cite{Br}. 
\end{remark}

\section{The density theorem}

After gathering all the ingredients we are now ready to prove our density theorem. We do so in two steps. First, we use the Kuznetsov formula to establish a bound for the generic part of the discrete spectrum. Second, we account for the missing pieces by hand.

\subsection{The generic contribution}\label{sec:gen}

As mentioned in the introduction the generic spectrum can be treated using a Kuznetsov formula following the strategy pioneered in \cite{Bl}. The argument was adapted to Siegel congruence subgroups of ${\rm Sp}_4(\mathbb{Z})$ in \cite{Man}. We will start by discussing the overall strategy of this argument for a quite general class of lattices $\Gamma$. Later we will restrict our attention to the paramodular groups $\Gamma_{\rm pa}(q)$, where we can carry out all the details.

\begin{remark}\label{mistake}
Unfortunately it seems that the proof of the density theorem given in \cite{Man} has the following gap. As we will see below the non-vanishing of the first Fourier coefficient $A_{\varpi}(\mathbf{1})$ for generic Siegel-Maa\ss\  cusp forms is crucial to the approach using the Kuznetsov formula. For the Siegel congruence subgroup this is stated in \cite[(6.3)]{Man} and is supposed to follow from \cite[Theorem~1.1]{CI}. However, \cite{CI} works with the paramodular group in place of the Siegel congruence subgroup. It is not clear to us if and with how much work the methods form \cite{CI} adapt to generic Siegel-Maa\ss\  cusp forms for the Siegel congruence subgroup. An alternative would be to establish the desired estimate conditional on the Conjecture of Lapid-Mao and some extensive local computations. See Conjecture~\ref{conj}, \eqref{con_2} and the discussion below.
\end{remark}

It is easy to see that the density result follows from the weighted estimate
\begin{equation}
	\sum_{\substack{\pi \mid X_{\Gamma},\\ \text{generic}}} \dim_{\mathbb{C}}V_{\Gamma}(\pi)\cdot Z^{2\sigma_{\pi}} \ll \mathcal{V}(\Gamma)^{1+\epsilon},\label{Z-version}
\end{equation}
for sufficiently large $Z$. We call the value of $Z$ necessary to establish the density hypothesis $Z_0(\Gamma)$. It is given by $Z_0(\Gamma) = \mathcal{V}(\Gamma)^{\frac{1}{2\sigma_{\mathbf{1}}}}= \mathcal{V}(\Gamma)^{\frac{1}{3}}$. If the estimate holds for larger values of $Z$, then one obtains a subconvex density theorem. For example we have $Z_0(\Gamma_{\rm pa}(q)) = q^{\frac{2}{3}}$, since $\mathcal{V}(\Gamma_{\textrm{pa}}(q))=q^2$.

We will now recall the Kuznetsov formula from \cite[Section~7]{Man} and modify it slightly to apply to all lattices $\Gamma$ in question. Note that one can alternative use the relative trace formula developed in \cite{Co}. 

For the following argument we assume that $U(\mathbb{Z})\subseteq \Gamma$. Given a function $E \colon \mathbb{R}^2_+\to \mathbb{C}$ with compact support and a parameter $X\in \mathbb{R}^2_+$ we set
\begin{equation}
	E^{(X)}(y_1,y_2) = E(X_1y_1,X_2y_2).\nonumber
\end{equation}
This function is then lifted to a function $F^{(X)}$ on ${\rm Sp}_4(\mathbb{R})$ by
\begin{equation}
	F^{(X)}(xyk) = \boldsymbol{\psi}(x)E^{(X)}({\rm y}(y)) \text{ for }x\in U(\mathbb{R}),\, y\in T(\mathbb{R}_+) \text{ and }k\in K_{\infty}.\nonumber
\end{equation}
We are now ready to define the Poincar\'e series
\begin{equation}
	P_{\Gamma}^{(X)}(g) = \sum_{\gamma\in U(\mathbb{Z})\backslash \Gamma}F^{(X)}(\gamma g), \text{ for }g\in {\rm Sp}_4(\mathbb{R}).\nonumber
\end{equation}
The usual unfolding trick shows that
\begin{equation}
	\langle \varpi, P_{\Gamma}^{(X)}\rangle_{\Gamma} = A_{\varpi}(\mathbf{1})\langle W_{\mu},E^{(X)}\rangle.\nonumber
\end{equation}
On the other hand we can compute the Fourier coefficient of the Poincar\'e series directly using the Bruhat decomposition. One gets
\begin{equation}
	\int_{U(\mathbb{Z})\backslash U(\mathbb{R})}P^{(X)}_{\Gamma}(xg)\overline{\boldsymbol{\psi}_{\infty}(x)}dx = \sum_{w\in W}\sum_{c\in \mathbb{N}^2}{\rm KL}_{\Gamma,w}(c;1,1)\cdot \int_{U_w(\mathbb{R})}F^{(X)}(c^{\ast}wxg)\overline{\boldsymbol{\psi}_{\infty}(x)}dx.\nonumber
\end{equation}
Taking the inner product of two Poincar\'e series and applying Parseval gives the following lemma:

\begin{lemma}
Let $Z\in \mathbb{R}_+$ and let $E\colon \mathbb{R}_+^2\to \mathbb{C}$ be compactly supported. Then we have
\begin{multline}
	{\sum_{\substack{\pi\mid X_{\Gamma},\\ \text{generic}}}}'\sum_{\varpi\in \mathcal{O}_{\Gamma}(\pi)}\frac{\vert A_{\varpi}(\mathbf{1})\vert^2}{\langle \varpi,\varpi\rangle_{\Gamma}}\cdot Z^{2\sigma_{\pi}} \ll_M \sum_{w\in W}  \sum_{c\in \mathbb{N}^2}\frac{{\rm KL}_{\Gamma,w}(c;1,1)}{c_1c_2}\\ \cdot{\rm y}(A_c)^{-\eta}\int_{T(\mathbb{R}_+)} \int_{U_w(\mathbb{R})}F^{(X)}(\iota^{-1}(X)^{-1}A_cwxy)\boldsymbol{\psi}_{\infty}^{(X^{-1})}(x^{-1})\overline{E({\rm y}(y))}dxd^{\ast}y,\nonumber
\end{multline}
for $A_c = \iota(X)c^{\ast}w\iota(X)^{-1}w^{-1}\in T(\mathbb{R}_+)$ and $X=(1,Z)$.
\end{lemma}
\begin{proof}
This follows directly from a version of \cite[Lemma~7.1]{Man} for $\Gamma$ after inserting \cite[Lemma~5.2]{Man} and dropping the continuous spectrum. Note that in \cite[Lemma~7.1]{Man} the factor $X^{2\eta} = Z^{2\eta_2}$ appears on the geometric side. However, the same factor enters the spectral side via the estimate \cite[Lemma~5.2]{Man} and the contributions cancel out.
\end{proof}

The geometric side can further be brought in the following form
\begin{lemma}\label{lm:geo_side}
For $Z\in \mathbb{R}_+$ we have
\begin{equation}
	{\sum_{\substack{\pi\mid X_{\Gamma},\\ \text{generic}}}}'\sum_{\varpi\in \mathcal{O}_{\Gamma}(\pi)}\frac{\vert A_{\varpi}(\mathbf{1})\vert^2}{\langle \varpi,\varpi\rangle}\cdot Z^{2\sigma_{\pi}} \ll_{M,\epsilon} \mathcal{N}(\Gamma) + Z^{\epsilon}\sum_{\substack{1\neq w\in W,\\ \text{relevant}}}S_{\Gamma}(w;Z), \nonumber
\end{equation}
for
\begin{align}
	S_{\Gamma}(s_1s_2s_1;Z) &= \sum_{\substack{c=(c_1,c_1)\in\mathbb{N}^2,\\ c_1\ll Z}}\frac{{\rm KL}_{\Gamma,w}(c;1,1)}{c_1^2},\nonumber\\ 
	S_{\Gamma}(s_2s_1s_2;Z) &= \sum_{\substack{c=(c_1,c_1^2)\in\mathbb{N}^2,\\ c_1\ll Z}}\frac{{\rm KL}_{\Gamma,w}(c;1,1)}{c_1^3}\text{ and }\nonumber\\ 
	S_{\Gamma}(s_1s_2s_1s_2;Z) &= \sum_{\substack{c=(c_1,c_2)\in \mathbb{N}^2,\\ c_1\ll Z,\\ c_2\ll Z^2}}\frac{{\rm KL}_{\Gamma,w}(c;1,1)}{c_1c_2}.\label{eq:def_long}
\end{align}
\end{lemma}
\begin{proof}
First, we note that for $(w,c)$ non-admissible the corresponding Kloosterman sum vanishes. Thus we drop those tuples from the geometric side. Furthermore, the contribution of $w=1$ is given by \eqref{eq:triv_KS}.

We consider the remaining cases. Let $w\in W$ be a non-trivial relevant Weyl element. The corresponding part of the geometric side of the Kuznetsov formula is estimated as follows. We first apply \cite[Lemma~3.1]{Man} as in \cite[Section~8]{Man} to truncate the $c$-sum. The result is the contribution
\begin{equation}
	\sum_{\substack{c\in \mathbb{N}^2,\\ c_1\ll Z,\\ c_2\ll Z^{\kappa(w)}}}\frac{{\rm KL}_{\Gamma,w}(c;1,1)}{c_1c_2} \cdot{\rm y}(A_c)^{-\eta}\int_{T(\mathbb{R}_+)} \int_{U_w(\mathbb{R})}F^{(X)}(\iota^{-1}(X)^{-1}A_cwxy)\boldsymbol{\psi}_{\infty}^{(X^{-1})}(x^{-1})\overline{E({\rm y}(y))}dxd^{\ast}y, \label{to_est}
\end{equation}
for $\kappa(s_1s_2s_1) = 1$ and $\kappa(s_2s_1s_2) = \kappa(s_1s_2s_1s_2) = 2.$ Now one uses \cite[Lemma~3.1 and Lemma~3.3]{Man} to estimate the orbital integrals as in \cite[Section~8]{Man}.\footnote{Note that \cite[Lemma~3.3]{Man} is stated only for $w=s_2s_1s_2$. However the statement remains true in general. For ${\rm Sp}_4$ this can be checked by hand in the remaining two cases. Otherwise the argument from \cite{A} applies here as well.} This gives 
\begin{equation}
	\left\vert\int_{T(\mathbb{R}_+)} \int_{U_w(\mathbb{R})}F^{(X)}(\iota^{-1}(X)^{-1}A_cwxy)\boldsymbol{\psi}_{\infty}^{(X^{-1})}(x^{-1})\overline{E({\rm y}(y))}dxd^{\ast}y\right\vert \ll_{E,\epsilon} {\rm y}(A_c)^{\eta(1+\epsilon)} \nonumber
\end{equation}
as on the bottom of \cite[p. 2071]{Man}. Thus \eqref{to_est} is 
\begin{equation}	
	\ll_{M,\epsilon} Z^{\epsilon}\cdot \sum_{\substack{c\in \mathbb{N}^2,\\ c_1\ll Z,\\ c_2\ll Z^{\kappa(w)}}}\frac{{\rm KL}_{\Gamma,w}(c;1,1)}{c_1c_2}.\nonumber
\end{equation}
The desired result follows after inserting the admissibility constraints for $c$ as recorded in Table~\ref{tab:my_label}.
\end{proof}

For the paramodular group $\Gamma=\Gamma_{\rm pa}(q)$ we can use the computations from Section~\ref{sec_KS} to deduce explicit estimates for the geometric side:
\begin{lemma}\label{geo_para}
We have
\begin{equation}
	S_{\Gamma_{\rm pa}(q)}(s_1s_2s_1;Z) = S_{\Gamma_{\rm pa}(q)}(s_2s_1s_2;Z)=S_{\Gamma_{\rm pa}(q)}(s_1s_2s_1s_2;Z) = 0, \nonumber
\end{equation}
for $Z\ll q^{1-\epsilon}$. Furthermore, if $q$ is prime and $Z\ll q^{2-\epsilon}$, then we have
\begin{align}
	S_{\Gamma_{\rm pa}(q)}(s_1s_2s_1;Z) &\ll q,\nonumber \\
	S_{\Gamma_{\rm pa}(q)}(s_2s_1s_2;Z) &= 0 \text{ and }\nonumber \\
	S_{\Gamma_{\rm pa}(q)}(s_1s_2s_1s_2;Z) &\ll \frac{Z^2}{q^{1+\frac{1}{4}}}.\nonumber 
\end{align}
\end{lemma}
\begin{proof}
If $Z\ll q^{1-\epsilon}$ everything follows from the definition of $S_{\Gamma_{\rm pa}(q)}(w;Z)$ together with Lemma~\ref{para_ram_KS}. Thus, we assume that $q$ is prime and that $Z\ll q^{2-\epsilon}$. This implies that $(c_1,q^{\infty}) =q$ for all moduli $c=(c_1,c_2)$ that appear in the contribution of relevant Weyl elements. With this simple observation the sums $S_{\Gamma_{\textrm{pa}}(q)}(\ast;Z)$ can be easily estimated. The case $S_{\Gamma_{\rm pa}(q)}(s_1s_2s_1s_2;Z)$ is most interesting. From the definition, see \eqref{eq:def_long}, and the factorization formula \eqref{fact_KS} we get
\begin{align}
	S_{\Gamma_{\rm pa}(q)}(s_1s_2s_1s_2;Z) &= \sum_{\substack{c=(c_1,c_2)\in \mathbb{N}^2,\\ c_1\ll Z,\\ c_2\ll Z^2}}\frac{{\rm KL}_{\Gamma_{\rm pa}(q),w}(c;1,1)}{c_1c_2} \nonumber \\
	&= \sum_{i=1,2,3} \frac{\vert {\rm KL}_{\Gamma_{\rm pa}(q),w}((q,q^i);1,N')\vert }{q^{i+1}}\sum_{\substack{c=(c_1,c_2)\in \mathbb{N}^2,\\ c_1\ll Z/q,\\ c_2\ll Z^2/q^i,\\ (c_1c_2,q)=1}}\frac{\vert{\rm KL}_{\Gamma_{\rm pa}(q),w}(c;1,N'')\vert}{c_1c_2}. \label{split_up}
\end{align}
Inserting Lemma~\ref{para_ram_KS} and Theorem~\ref{unram_KS} yields
\begin{equation}
	S_{\Gamma_{\rm pa}(q)}(s_1s_2s_1s_2;Z) \ll  \sum_{i=1,2,3}q^{1+\epsilon-\frac{i+1}{2}}\sum_{\substack{c=(c_1,c_2)\in \mathbb{N}^2,\\ c_1\ll Z/q,\\ c_2\ll Z^2/q^i,\\ (c_1c_2,q)=1}} \frac{(c_1,c_2)^{\frac{1}{2}}}{c_1^{\frac{1}{2}}c_2^{\frac{1}{4}}}.\nonumber
\end{equation}
Estimating the $c$-sum is now routine. Indeed, for $i\in\{1,2,3\}$ we have
\begin{equation}
	\sum_{\substack{c=(c_1,c_2)\in \mathbb{N}^2,\\ c_1\ll Z/q,\\ c_2\ll Z^2/q^i,\\ (c_1c_2,q)=1}} \frac{(c_1,c_2)^{\frac{1}{2}}}{c_1^{\frac{1}{2}}c_2^{\frac{1}{4}}} = \sum_{\substack{a\leq Z/q, \\ (a,q)=1}} a^{-\frac{1}{4}}\sum_{\substack{c_1\ll \frac{Z}{q},\\ a\mid c_1,\\ (c_1,q)=1}}c_1^{-\frac{1}{2}}\sum_{\substack{c_2\ll \frac{Z^2}{q^i},\\ a\mid c_2\\ (c_2/a,qc_1/a)=1}}c_2^{-\frac{1}{4}} \ll \frac{Z^2}{q^{\frac{1}{2}+\frac{3i}{4}}}.\nonumber
\end{equation}
The result follows, with the bottleneck being the contribution of $i=1$.
\end{proof}

To complement the geometric estimates above we need to handle the spectral side of the Kuznetsov formula. More precisely, we need to show that
\begin{equation}
	\sum_{\varpi\in V_{\Gamma}(\pi)}\vert A_{\varpi}(\mathbf{1})\vert^2 \nonumber
\end{equation}
is not to small. Establishing this estimate is not an easy task. Assuming Conjecture~\ref{conj} this reduces to a purely local computation, which can in principle be carried out for many families of congruence subgroups. However, these local computations can be rather involved.  Luckily, in the case of $\Gamma_{\rm pa}(q)$ with $q$ square-free, the desired result is available in \cite{CI}.

\begin{lemma}\label{lm:spec_side}
Let $q\in \mathbb{N}$ be square-free and let $\pi$ be a generic cuspidal automorphic representation. Then we have
\begin{equation}
	\mathcal{V}(\Gamma_{\rm pa}(q))^{\epsilon}\sum_{\varpi\in \mathcal{O}_{\Gamma_{\rm pa}(q)}(\pi)}\frac{\vert A_{\varpi}(\mathbf{1})\vert^2}{\langle \varpi,\varpi\rangle} \gg_{\epsilon}  \frac{\mathcal{N}(\Gamma_{\rm pa}(q))}{\mathcal{V}(\Gamma_{\rm pa}(q))}\cdot   \dim_{\mathbb{C}} V_{\Gamma_{\rm pa}(q)}(\pi). \label{goal_spec}
\end{equation}
\end{lemma}
\begin{proof}
Recall from \eqref{eq:N_for_para} that $\mathcal{N}(\Gamma_{\rm pa}(q))=q$. Further, we observe that $\dim V_{\Gamma_{\rm pa}(q)}(\pi) \ll \mathcal{V}(\Gamma_{\rm pa}(q))^{\epsilon}$. (The latter can be easily extracted from \cite[Table~A.12]{RS} and the local nature of $V_{\Gamma_{\rm pa}(q)}(\pi)$.) From Theorem~\ref{CI_th} we get
\begin{equation}
	\frac{\mathcal{N}(\Gamma_{\rm pa}(q))}{\mathcal{V}(\Gamma_{\rm pa}(q))}\cdot \frac{1}{\Lambda^q(1,\pi, \textrm{Ad})}\ll \sum_{\varpi\in \mathcal{O}_{\Gamma_{\rm pa}(q)}(\pi)}\frac{\vert A_{\varpi}(\mathbf{1})\vert^2}{\langle \varpi,\varpi\rangle}.
\end{equation}
We conclude by applying \cite[Theorem~2]{Li} to obtain an upper bound for $\Lambda^q(1,\pi, \textrm{Ad})$.
\end{proof}

Combining the results of this section gives the following theorem for the paramodular group:

\begin{theorem}\label{th:gen_ests}
Let $q$ be square-free. We have
\begin{equation}
	N_{\Gamma_{\rm pa}(q)}^{\rm gen}(\sigma;M) \ll_{M,\epsilon}\mathcal{V}(\Gamma_{\rm pa}(q))^{1-\frac{2\sigma}{3}(1+\delta)+\epsilon}.\label{dens_goal}
\end{equation}
for $\delta=\frac{1}{2}$. If $q$ is prime, then we have \eqref{dens_goal} with $\delta=\frac{11}{16}$. Furthermore, assuming Conjecture~\ref{conj} and \eqref{con_2} yields \eqref{dens_goal} for arbitrary $q$ with $\delta=\frac{1}{2}$.
\end{theorem}
\begin{proof}
We first observe that, if \eqref{Z-version} holds with $Z\asymp\mathcal{V}(\Gamma_{\rm pa}(q))^{\alpha-\epsilon}$, then we obtain \eqref{dens_goal} with $\delta = 3\alpha-1$. In particular, we need $\alpha\geq \frac{1}{3}$ to achieve the density hypothesis. We continue by combining Lemma~\ref{lm:geo_side} and Lemma~\ref{lm:spec_side} to obtain
\begin{equation}
	{\sum_{\substack{\pi\mid X_{\Gamma_{\rm pa}(q)},\\ \text{generic}}}}\dim V_{\Gamma_{\rm pa}(q)}(\pi)\cdot Z^{2\sigma_{\pi}} \ll_M \mathcal{V}(\Gamma_{\rm pa}(q))^{1+\epsilon} + \frac{\mathcal{V}(\Gamma_{\rm pa}(q))^{1+\epsilon}}{\mathcal{N}(\Gamma_{\rm pa}(q))}\sum_{\substack{1\neq w\in W,\\ \text{relevant}}}S_{\Gamma_{\rm pa}(q)}(w;Z) .\nonumber
\end{equation}
This holds for $q$ square-free. In general we need to assume Conjecture~\ref{conj} and \eqref{con_2} in order to replace Lemma~\ref{lm:spec_side}. We conclude the proof by using Lemma~\ref{geo_para} to estimate $S_{\Gamma_{\rm pa}(q)}(w;Z)$. Let us briefly explain this. First, if $\alpha=\frac{1}{2}$ (i.e. $Z=q$), then $S_{\Gamma_{\rm pa}(q)}(w;Z)=0$ for all admissible $w\neq 1$. (This applies to all $q$ not just square-free ones.) Second, if $q$ is prime, then we can take $\alpha=\frac{9}{16}$ (i.e. $Z=q^{1+\frac{1}{8}}$).
\end{proof}

\begin{remark}
We expect versions of Theorem~\ref{th:gen_ests} to hold in great generality. Indeed the main global tool, which is the Kuznetsov formula, is very flexible. Thus one is left with the problem of estimating the first Fourier coefficient on the spectral side and the ramified Kloosterman sums on the geometric side. The spectral problem can, at least conditionally on Conjecture~\ref{conj}, be reduced to the purely local problem of establishing \eqref{con_2}. The geometric problem also reduces to a local problem. Namely to a suitable estimate for the ramified Kloosterman sets. We believe that for \eqref{dens_goal} with $\delta=0$ only very weak bounds for the ramified Kloosterman sets suffice. 
\end{remark}

\subsection{The CAP and residual contribution}

We will first decompose the counting function $N_{\Gamma}(\sigma;M)$ as
\begin{equation}
N_{\Gamma}(\sigma;M) = \sum_{\ast \in \{ \mathbf{G}, \mathbf{Y}, \mathbf{Q}, \mathbf{P}, \mathbf{B}, \mathbf{F}\}} N_{\Gamma}^{\mathbf{(}\boldsymbol{\ast}\mathbf{)}}(\sigma;M), \nonumber
\end{equation}
where
\begin{equation}
N_{\Gamma}^{\mathbf{(}\boldsymbol{\ast}\mathbf{)}}(\sigma;M) = \sum_{\psi \text{ of type }(\ast)}\sum_{\substack{\pi\in \Pi_{\psi},\\ \langle \cdot,\pi\rangle=\epsilon_{\psi}}}\dim V_{\Gamma}(\pi).\nonumber
\end{equation}
Note that $\dim V_{\Gamma}(\pi) = 0$ unless $\pi\mid X_{\Gamma}$.

\begin{theorem}\label{th_CAP_type_B}
Let $q>1$ be arbitrary. The contribution of CAP-representations of type \textbf{(B)} and \textbf{(Q)} to the spectrum of $X_{\Gamma_{\rm pa}(q)}$ is 
\begin{equation}
	N_{\Gamma_{\rm pa}(q)}^{\mathbf{(B)}}(\sigma;M) =N_{\Gamma_{\rm pa}(q)}^{\mathbf{(Q)}}(\sigma;M)= 0. \nonumber
\end{equation}
Furthermore, the contribution of CAP-representations of type \textbf{(P)} is 
\begin{equation}
	N_{\Gamma_{\rm pa}(q)}^{\mathbf{(P)}}(\sigma;M) \ll q^{1+\epsilon}. \label{ess_shar}
\end{equation}
\end{theorem}
\begin{proof}
The first part follows directly from Lemma~\ref{lm:pack_B_and_Q}. We turn towards the type \textbf{(P)} contribution. In this case an application of Lemma~\ref{lm:pack_P} yields
\begin{equation}
N_{\Gamma}^{\mathbf{(P)}}(\sigma;M) \ll q^{\epsilon} \sharp\{ \psi = (\mu\boxtimes 1)\boxplus (\mathbf{1}\boxtimes\nu(2))\in \boldsymbol{\psi}_2(G,\mathbf{1})\colon c(\mu)\mid q\}.\nonumber
\end{equation}
for $\sigma\leq \frac{1}{2}$. Thus we are essentially counting cuspidal automorphic representations $\mu$ of ${\rm GL}_2$ with trivial central character, conductor dividing $q$ (and bounded spectral parameter). The result follows from the appropriate Weyl law, see for example \cite[Theorem~1.1]{Do}.
\end{proof}

\subsection{The endgame}\label{sec_endgame}

We are finally ready to put all the pieces together and assemble the final density theorem.

\begin{theorem}\label{th:detail_main}
For square-free $q$, we have the density result
\begin{equation}
	N_{\Gamma_{\rm pa}(q)}(\sigma;M) \ll_{M,\epsilon} \mathcal{V}(\Gamma_{\rm pa}(q))^{1-\frac{2}{3}\sigma(1+\frac{1}{2})+\epsilon}+1. \nonumber
\end{equation} 
More precisely, for each $\ast\in\{\mathbf{G}, \mathbf{Y}, \mathbf{B}, \mathbf{P},\mathbf{Q}, \mathbf{F}\}$ there is $0\leq \sigma(\ast)\leq \frac{3}{2}$ such that $N_{ \Gamma_{\rm pa}(q)}^{\mathbf{(\ast)}}(\sigma;M)=0$ for $\sigma>\sigma(\ast)$ and for $\sigma \leq \sigma(\ast)$ we have good estimates of the form $$N_{\Gamma_{\rm pa}(q)}(\sigma;M) \ll_{M,\epsilon}\mathcal{V}(\Gamma_{\rm pa}(q))^{C_{(\ast)}(\sigma,\epsilon)}.$$ We summarize the values for $\sigma(\ast)$ and $C_{(\ast)}(\sigma,\epsilon)$ in the following table:
\begin{center}
{\tabulinesep=1.2mm
\begin{tabu}{| c | c | c | c |}
	\hline
	$\ast$ & Condition on $q$ & $\sigma(\ast)$ & $C_{(\ast)}(\sigma,\epsilon)+\epsilon$ \\
	\hline
	\textbf{G}  & square-free & $\frac{9}{22}$ &  $1-\frac{2}{3}\sigma(1+\frac{1}{2})+\epsilon$ \\
	\hline
	\textbf{G}  & prime & $\frac{9}{22}$ & $1-\frac{2}{3}\sigma(1+\frac{11}{16})+\epsilon$\\
	\hline
	\textbf{Y}  & square-free & $\frac{7}{64}$ & $1-\frac{2}{3}\sigma(1+\frac{1}{2})\epsilon$ \\
	\hline
	\textbf{Y}  & prime  & $\frac{7}{64}$ & $1-\frac{2}{3}\sigma(1+\frac{11}{16})+\epsilon$ \\
	\hline
	\textbf{B}  & none & $0$ & - \\
	\hline
	\textbf{P}  & none & $\frac{1}{2}$ & $1+\epsilon$ \\
	\hline
	\textbf{Q}  & none & $0$ & - \\
	\hline
	\textbf{F} & none & $\frac{3}{2}$ & $0$ \\
	\hline
\end{tabu}}
\end{center}
\end{theorem}
\begin{proof}
We start by observing that, by Lemma~\ref{tempered_pack} (in particular \eqref{eq:in_proof}), we have
\begin{equation}
	N_{\Gamma_{\rm pa}(q)}^{\mathbf{(G)}}(\sigma;M) + N_{\Gamma_{\rm pa}(q)}^{\mathbf{(Y)}}(\sigma;M) = N^{\rm gen}_{\Gamma_{\rm pa}(q)}(\sigma;M).\nonumber
\end{equation}
The desired bound follows from Theorem~\ref{th:gen_ests}. The vanishing results can be deduced from absolute bounds towards the Ramanujan Conjecture for ${\rm GL}_4$ and are given in \cite[Corollary~3.4]{AS}. The results for the CAP-contributions are provided in Theorem~\ref{th_CAP_type_B}.
\end{proof}

This result directly implies Theorem~\ref{main_theorem} (and more).\\

\textbf{Acknowledgement:} I would like to thank V. Blomer for his encouragement while working on the project. Further I would like to thank R. Schmidt for patiently answering my questions related to Arthur parameters for ${\rm GSp}_4$. We also thank the anonymous referee for a very carefully reading the manuscript and for making many useful suggestions.

\end{document}